\RequirePackage{silence}
\documentclass{siamart220329}

\usepackage{amssymb}
\usepackage{mathtools}
\usepackage{booktabs}
\usepackage{enumitem}
\usepackage{flafter}
\usepackage{algpseudocode}

\makeatletter
\AtBeginDocument{%
  \@ifundefined{refstepcounter@optarg}{}{%
    \def\refstepcounter@optarg[#1]#2{%
      \cref@old@refstepcounter{#2}%
      \cref@constructprefix{#2}{\cref@result}%
      \@ifundefined{cref@#1@alias}%
        {\def\@tempa{#1}}%
        {\def\@tempa{\csname cref@#1@alias\endcsname}}%
      \protected@edef\cref@currentlabel{%
        [\@tempa][\arabic{#2}][\cref@result]%
        \csname p@#2\endcsname\csname the#2\endcsname}}%
  }%
}
\makeatother

\newsiamthm{question}{Question}
\newsiamremark{remark}{Remark}

\newcommand{\R}{\mathbb{R}}
\newcommand{\N}{\mathbb{N}}
\newcommand{\ip}[2]{\langle #1,#2\rangle}
\newcommand{\norm}[1]{\lVert #1\rVert}
\newcommand{\Pt}{\widetilde\Pi}
\newcommand{\PiG}{\Pi^{\mathrm G}}
\newcommand{\half}{\tfrac12}

\DeclareMathOperator{\diag}{diag}
\DeclareMathOperator{\spec}{\rho}

\DeclareSymbolFont{matrixsf}{OT1}{cmss}{m}{n}
\SetSymbolFont{matrixsf}{bold}{OT1}{cmss}{bx}{n}
\DeclareMathSymbol{\matA}{\mathalpha}{matrixsf}{`A}
\DeclareMathSymbol{\matB}{\mathalpha}{matrixsf}{`B}
\DeclareMathSymbol{\matE}{\mathalpha}{matrixsf}{`E}
\DeclareMathSymbol{\matI}{\mathalpha}{matrixsf}{`I}
\DeclareMathSymbol{\matK}{\mathalpha}{matrixsf}{`K}
\DeclareMathSymbol{\matP}{\mathalpha}{matrixsf}{`P}
\DeclareMathSymbol{\matR}{\mathalpha}{matrixsf}{`R}
\begingroup
\lccode`~=`A \lowercase{\endgroup\def~{\matA}}
\mathcode`A="8000
\begingroup
\lccode`~=`B \lowercase{\endgroup\def~{\matB}}
\mathcode`B="8000
\begingroup
\lccode`~=`E \lowercase{\endgroup\def~{\matE}}
\mathcode`E="8000
\begingroup
\lccode`~=`I \lowercase{\endgroup\def~{\matI}}
\mathcode`I="8000
\begingroup
\lccode`~=`K \lowercase{\endgroup\def~{\matK}}
\mathcode`K="8000
\begingroup
\lccode`~=`P \lowercase{\endgroup\def~{\matP}}
\mathcode`P="8000
\begingroup
\lccode`~=`R \lowercase{\endgroup\def~{\matR}}
\mathcode`R="8000

\title{The symmetric V-cycle can diverge under the multigrid axioms
for cell-centred discretisations\thanks{Preprint, \today.}}
\author{Ming Hei Wong\thanks{%
Department of Mathematics, University of Tennessee, Knoxville, TN 37996,
USA (\email{mwong4@vols.utk.edu}).}}
\headers{The symmetric V-cycle can diverge under cell-centred axioms}{M.~H.~Wong}

\begin{document}
\maketitle
\begin{abstract}
The axiomatic convergence theory for multigrid methods applied to cell-centred
finite-difference and finite-volume discretisations rests on two hypotheses: an
\emph{imbalanced Galerkin condition} $(G3)$, which states that
$R_{\ell-1}A_\ell P_{\ell-1}=2A_{\ell-1}$ with $R_{\ell-1}=\half
P_{\ell-1}^{T}$, and a \emph{weak approximation property} $(A2)_\alpha$ of
Bramble type.  Under these hypotheses, together with Richardson smoothing, the
symmetric W-cycle and the variable V-cycle are known to be uniformly
convergent, while the uniform convergence of the standard symmetric V-cycle
has remained open.  We answer the question in the negative, in a strong sense,
by two constructions.  First, for every smoothing count $m\in\N$ we exhibit
hierarchies of every depth satisfying $(G3)$, Richardson admissibility with
$C_R=1$, and $(A2)_\alpha$ for every $\alpha\in(0,1]$ with the sharp
level-independent constant $C_{A2}^{2}=4m$, whose symmetric $V(m,m)$-cycle
error operator has spectral radius $\theta_m(1+2\theta_m)>1$ already on
three levels, where $\theta_m=(1-\frac1{4m})^{2m}$, and spectral radius
growing geometrically with the depth; the family shows that any
smoothing-count threshold $m_0$ that could restore uniform V-cycle
convergence must grow at least quadratically in $C_{A2}$.  Second, and in
contrast with the synthetic algebraic construction, we prove that the same
failure occurs in a completely standard discretisation: the cell-centred finite-volume hierarchy for a
one-dimensional diffusion equation with a mesh-aligned coefficient jump
$1\,{:}\,\kappa$ and harmonic (Samarskii) interface averaging satisfies
$(G3)$ exactly and $(A2)_{1/2}$ with a level-independent constant
$C_{A2}=O(\kappa)$, yet for every $\kappa\ge3$ its symmetric $V(1,1)$-cycle
with \emph{any} admissible Richardson parameter---including the optimal
one---diverges geometrically in the number of levels.  In both constructions,
the W-cycle remains uniformly contractive, so the hypotheses separate the two
cycles.  All claims are verified numerically, to machine precision where
exact, by independent assembly of the cycle operators.
\end{abstract}

\begin{keywords}
multigrid, V-cycle, W-cycle, cell-centred finite volumes, Galerkin
condition, weak approximation property, counterexample
\end{keywords}

\begin{MSCcodes}
65N55, 65F10, 65N08
\end{MSCcodes}

\section{Introduction}\label{sec:intro}

Let $A_\ell\in\R^{n_\ell\times n_\ell}$, $\ell=0,1,\dots,L$, be a hierarchy of
symmetric positive definite matrices connected by full-rank prolongations
$P_{\ell-1}\in\R^{n_\ell\times n_{\ell-1}}$ and restrictions
$R_{\ell-1}\in\R^{n_{\ell-1}\times n_\ell}$, and let each level carry a
convergent smoother.  The classical multigrid cycles---V, W, and their
relatives---are obtained from one another by changing only the number $p$ of
recursive coarse-level calls: $p=1$ gives the V-cycle, $p=2$ the W-cycle
\cite{hackbusch,tos}.  For concreteness, Algorithm~\ref{alg:cycle} records
the symmetric cycle in the axiomatic setting: the data are the hierarchy
$(A_\ell,P_\ell,R_\ell)$ together with Richardson parameters $\Lambda_\ell$
(Definition~\ref{def:richardson} below), and the linear system $A_Lu=g_L$ is
solved by iterating $u\leftarrow\mathrm{MG}_p(L,g_L,u)$; the V-cycle and
W-cycle differ in a single line---one recursive coarse call against
two---and everything in this paper turns on that difference.  A central goal
of multigrid theory since the 1980s has
been to identify checkable hypotheses on $(A_\ell,P_\ell,R_\ell)$ and the
smoother under which these cycles contract uniformly in the number of
levels.

\begin{algorithm}[t]
\caption{Axiomatic symmetric cycle,
$u\leftarrow\mathrm{MG}_{p}(\ell,g_\ell,u)$: $p=1$ is the V-cycle, $p=2$
the W-cycle.}
\label{alg:cycle}
\begin{algorithmic}[1]
\If{$\ell=0$}
  \State \Return $A_0^{-1}g_0$
  \Comment{exact coarsest-level solve ($E_0=0$)}
\EndIf
\For{$i=1,\dots,m$}
  \State $u\leftarrow u+\Lambda_\ell^{-1}\,(g_\ell-A_\ell u)$
  \Comment{Richardson pre-smoothing}
\EndFor
\State $r_{\ell-1}\leftarrow R_{\ell-1}(g_\ell-A_\ell u)$
\Comment{restrict the residual}
\State $e_{\ell-1}\leftarrow0$
\For{$s=1,\dots,p$}
  \State $e_{\ell-1}\leftarrow\mathrm{MG}_{p}(\ell-1,\,r_{\ell-1},\,e_{\ell-1})$
  \Comment{$p$ recursive coarse calls}
\EndFor
\State $u\leftarrow u+P_{\ell-1}e_{\ell-1}$
\Comment{prolongate and correct}
\For{$i=1,\dots,m$}
  \State $u\leftarrow u+\Lambda_\ell^{-1}\,(g_\ell-A_\ell u)$
  \Comment{Richardson post-smoothing}
\EndFor
\State \Return $u$
\end{algorithmic}
\end{algorithm}

In the \emph{variational} (or Galerkin) setting, where the restrictions are
$R_{\ell-1}=P_{\ell-1}^{T}$ and the coarse matrices are the Galerkin
products $A_{\ell-1}=P_{\ell-1}^{T}A_\ell P_{\ell-1}$, the question has a
special character: the coarse-grid correction is then the
$A_\ell$-orthogonal projection onto the complement of the coarse space,
every cycle is a product of energy contractions, and the symmetric V-cycle error
operator automatically satisfies $\norm{E_\ell}_{A_\ell}<1$ on every fixed
hierarchy \cite{mccormick,xu-zikatanov}: divergence is impossible, only the
\emph{uniformity} of the contraction numbers can degenerate, and a long line
of work has pushed uniform V-cycle bounds far below full elliptic
regularity \cite{braess-hackbusch,bpwx,brenner-vcycle,yserentant}.

Cell-centred finite-difference and finite-volume discretisations break this
variational structure in a specific, well-understood way.  With the natural
piecewise-constant prolongation $P$ (each coarse cell is the union of fine
cells, and a coarse value is copied to its children) and the mass-conserving
restriction $R=\half P^{T}$ in one space dimension, the rediscretised coarse
operator $A_{\ell-1}$ does \emph{not} equal $RA_\ell P$; instead the hierarchy
satisfies the \emph{imbalanced Galerkin condition}
\begin{equation}\label{eq:g3-intro}
  R_{\ell-1}A_\ell P_{\ell-1}=2\,A_{\ell-1},
\end{equation}
the factor $2$ being forced by the scaling of piecewise constants
(Figure~\ref{fig:mass}).  This
relation is the foundation of the convergence theory of Bramble, Ewing,
Pasciak and Shen \cite{beps} for cell-centred multigrid: combining
\eqref{eq:g3-intro} with a weak approximation property in the scale introduced
by Bramble and Pasciak \cite{bramble-book,bramble-pasciak}, they proved
uniform convergence of the symmetric W-cycle (for sufficiently many smoothing
steps) and uniform preconditioning bounds for the variable V-cycle.  The
uniform convergence of the \emph{standard} symmetric V-cycle under these
hypotheses was left open, and has remained open: the recent
textbook account of the axiomatic theory \cite{sww} states it explicitly as
an open problem, both in its abstract form
(``the abstract, axiomatic convergence of the standard V-cycle under
condition $(G3)$ remains an open question''
\cite[Remark~6.9.16]{sww}) and again for the cell-centred hierarchy itself
\cite[Remark~8.6.10]{sww}.  Positive results for cell-centred V-cycles
\cite{kwak,kwak-lee,mohr-wienands} all modify the transfer operators so as to
restore an exact or asymptotically exact variational structure, rather than
analyse the natural hierarchy \eqref{eq:g3-intro}.

\begin{figure}[t]
\centering
\includegraphics[width=\linewidth]{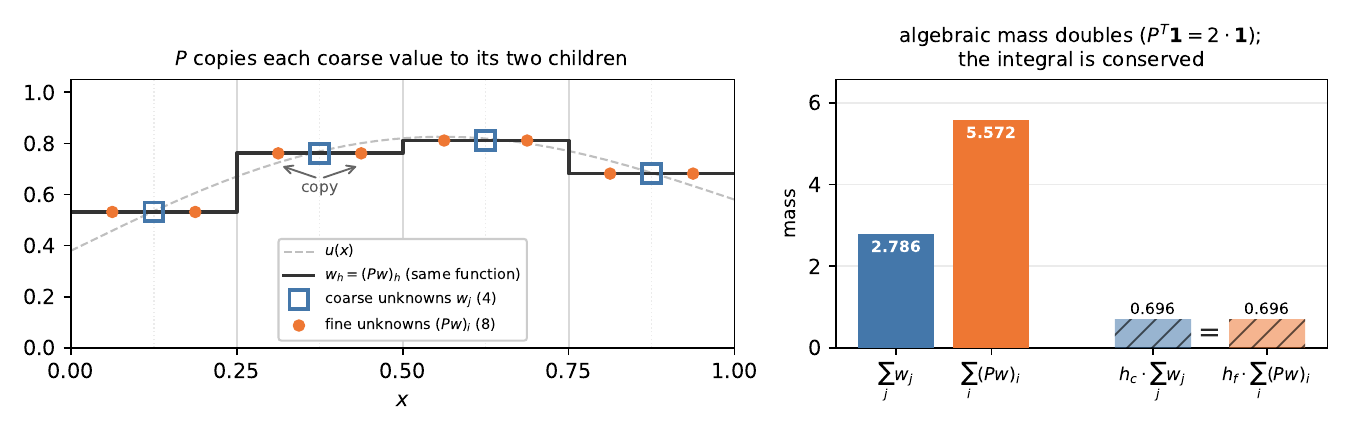}
\caption{The origin of the factor $2$ in \eqref{eq:g3-intro}, illustrated
on $(0,1)$ with four coarse and eight fine cells.  Left: the cell-centred
prolongation $P$ copies each coarse value to its two children, so the
piecewise-constant \emph{function}---and with it its integral---is
unchanged, but every unknown is counted twice.  Right: the algebraic mass
of the unknown vector therefore doubles, $\sum_i(Pw)_i=2\sum_jw_j$
(equivalently $P^{T}\mathbf 1=2\cdot\mathbf 1$), so the mass-conserving
cell-average restriction is $R=\half P^{T}$.  On the energy side, the
jumps of the unchanged function now sit across cells of half the width,
which doubles its discrete Dirichlet energy: in matrix form
$P^{T}A_\ell P=4A_{\ell-1}$ (Proposition~\ref{prop:g3-ode}).  The two
effects combine to the imbalance factor $2$:
$RA_\ell P=\half P^{T}A_\ell P=2A_{\ell-1}$.}
\label{fig:mass}
\end{figure}

This paper settles the question in the negative.  Under the hypotheses of the
cell-centred theory---condition \eqref{eq:g3-intro}, Richardson smoothing with
admissible damping, and the $\alpha$-weak approximation property with
level-independent constants---the symmetric V-cycle need not converge at all:
its spectral radius can exceed $1$ on three levels and can grow geometrically
with the depth of the hierarchy.  The mechanism is geometric and simple to
state: under \eqref{eq:g3-intro} an exact coarse solve \emph{reflects} the
error, in the energy inner product, across the orthogonal complement of the
coarse space instead of projecting it out, and the
V-recursion---unlike the W-recursion---propagates the sign of that
reflection.  The failure is not an artefact of contrived
matrices: it occurs in the textbook cell-centred finite-volume discretisation
of a one-dimensional high-contrast diffusion problem with harmonic interface
averaging, with the optimal Richardson parameter, while the W-cycle on the
same hierarchy contracts uniformly.

\subsection{Main results}\label{ssec:main}

Throughout, $E_\ell^{(p,m)}$ denotes the error propagation operator of the
symmetric cycle with $m$ pre- and $m$ post-smoothing steps and $p$ recursive
coarse calls (Section~\ref{sec:framework}); $E^{V}_\ell:=E^{(1,m)}_\ell$ and
$E^{W}_\ell:=E^{(2,m)}_\ell$.  The hypotheses $(G3)$, $(R)$, $(A2)_\alpha$ are
stated precisely in Definitions \ref{def:g3}--\ref{def:a2}.

\begin{theorem}[Example 1; Section~\ref{sec:ex2}]\label{thm:main-ex2}
For every $m\in\N$ there is a hierarchy of every depth $L\in\N$ satisfying
$(G3)$ with $r=\half$, Richardson admissibility with $C_R=1$, and
$(A2)_\alpha$ for every $\alpha\in(0,1]$ with the sharp level-independent
constant $C_{A2}^{2}=4m$, such that, with
$\theta_m=\bigl(1-\frac1{4m}\bigr)^{2m}>\frac12$,
\[
  \spec\bigl(E^{V}_\ell\bigr)
  =\theta_m\,\frac{(2\theta_m)^{\ell}-1}{2\theta_m-1}
  \;\xrightarrow[\;\ell\to\infty\;]{}\;\infty ;
\]
in particular $\spec(E^{V}_2)=\theta_m(1+2\theta_m)>1$ already on three
levels.  Moreover, on the same family the W-cycle is uniformly contractive,
$\spec(E^{W}_\ell)\le\theta_m<e^{-1/2}$ for all $\ell$ and $m$, at the same
smoothing count.  Finally, if the $V(\mu,\mu)$-cycle on this family is
required to contract uniformly in the depth, then necessarily
$\mu>\frac{\ln 3}{2}\,\bigl(C_{A2}^{2}-1\bigr)$: any smoothing-count threshold
that could rescue the V-cycle must grow at least quadratically in $C_{A2}$.
\end{theorem}

\begin{theorem}[Example 2; Section~\ref{sec:ex3}]\label{thm:main-ex3}
Fix $\kappa\ge3$ and let $(A_\ell)_{\ell\ge0}$ be the cell-centred
finite-volume hierarchy on dyadic meshes of $(-1,1)$ for the equation
$-(Du')'=f$, $u(\pm1)=0$, with $D=1$ on $(-1,0)$, $D=\kappa$ on $(0,1)$, and
harmonic (Samarskii) edge averaging, coarsened by aggregation of adjacent
cell pairs with $R=\half P^{T}$.  Then:
\begin{enumerate}[label=(\roman*),itemsep=1pt,topsep=2pt]
\item the hierarchy satisfies $(G3)$ exactly on every level;
\item it satisfies $(A2)_{1/2}$, and hence $(A2)_\alpha$ for every
  $\alpha\in(0,\half]$, with a level-independent constant
  $C_{A2}=O(\kappa)$;
\item Richardson smoothing is admissible with $C_R=2$, e.g.\ with
  $\Lambda_\ell=4\kappa h_\ell^{-2}$;
\item for \emph{every} admissible choice of Richardson parameters
  $\Lambda_\ell\ge\spec(A_\ell)$---in particular for the optimal choice
  $\Lambda_\ell=\spec(A_\ell)$---the symmetric $V(1,1)$-cycle diverges
  geometrically:
  \[
    \spec\bigl(E^{V}_\ell\bigr)\;\ge\;(q_2+1)\,\mu_\kappa^{\,\ell-2}-1
    \;\xrightarrow[\;\ell\to\infty\;]{}\;\infty,
    \qquad
    \mu_\kappa=2\Bigl(1-\frac{2}{3\kappa}\Bigr)^{2}>1,
  \]
  with $q_2+1>0$.  Numerically, for $\kappa=100$ the spectral radius
  of the $V(1,1)$ operator passes $2.5$ at level $7$ ($n=128$ cells) and
  thereafter grows per level by a factor decreasing towards
  $\mu_\kappa\approx1.98$ ($2.13$ at level $9$, $1.99$ at level $12$), while
  $\spec(E^{W}_\ell)\le0.993$ on all levels tested; uniform contractivity of
  the W-cycle on this hierarchy, for every $m\ge1$, follows from
  (i)--(iii) via \cite[Theorem~6.9.15]{sww} (Remark~\ref{rem:w-ode}).
\end{enumerate}
\end{theorem}

The two constructions share one mechanism, made transparent by condition
$(G3)$.  Under \eqref{eq:g3-intro}, the coarse-grid correction operator is
exactly twice the orthogonal projection $\PiG_\ell$ onto the coarse space with
respect to the $A_\ell$ inner product (Lemma~\ref{lem:structure}).  Thus the
two-level correction $I-2\PiG_\ell$ is an isometric \emph{reflection} in that
inner product---an exact coarse solve never decreases the energy of the error.
The W-cycle survives because its recursion squares the inner error operator,
destroying the sign of the reflection; the V-cycle composes with it linearly,
and the factor $2$ in \eqref{eq:g3-intro} then amplifies the coarse defect by
up to $2$ per level, along an explicitly tracked vector in both examples.

We emphasise three points of scope.  First, the counterexamples concern the
\emph{fixed} symmetric V-cycle; they are entirely consistent with the known
uniform bounds for the variable V-cycle \cite{beps,bramble-book} and for
Krylov-accelerated (K-)cycles \cite{notay-vassilevski}.  Second,
Theorem~\ref{thm:main-ex2} shows that uniform V-cycle convergence fails for
every smoothing count $\mu\lesssim\frac{\ln3}{2}C_{A2}^2$, but our examples
contract again once $\mu$ is much larger than $C_{A2}^{2}$; whether some
superquadratic threshold $m_0(C_R,C_{A2},\alpha)$ restores uniform V-cycle
contraction remains an interesting open problem
(Section~\ref{sec:discussion}).  Third,
Example~2 places the failure inside the cell-centred finite-volume class of
\cite[Chapter~8]{sww}, though with a variable, mesh-aligned coefficient: no
convergence theorem based on the axioms alone can cover the natural
cell-centred hierarchy; for the \emph{constant-coefficient} instance behind
\cite[Remark~8.6.10]{sww} itself we claim no divergence---numerically its
V-cycle contracts with spectral radii below $0.35$ on all levels
tested---so deciding it requires information beyond the axioms
(Section~\ref{sec:discussion}).

\subsection{The axiomatic question and the reading of the
hypotheses}\label{ssec:axiomatic}

Because the question we answer is an axiomatic one, it is worth being
precise about its shape.  The V-cycle itself is never abstract: it is the
concrete recursion \eqref{eq:mg-recursion} with $p=1$.  What is axiomatic is
the convergence \emph{theory} built around it, in the tradition of
\cite{bramble-book,bramble-pasciak,bpx} and developed in textbook form in
\cite{sww}: theorems of the schema ``let $(A_\ell,P_\ell,R_\ell,
\Lambda_\ell)$ be \emph{any} hierarchy satisfying the listed conditions;
then the cycle contracts uniformly in the number of levels''.  The
hypotheses are the entire interface---such a theorem is not permitted to know
whether the matrices arise from a PDE---and they divide the labour in a
definite way: $(R)$ carries the whole notion of smoothing strength (one
Richardson sweep damps the eigencomponent at eigenvalue $\lambda$ by
$1-\lambda/\Lambda_\ell$, so ``oscillatory'' means large energy relative to
$\spec(A_\ell)$ and nothing else); $(G3)$ governs the communication between
levels; and $(A2)_\alpha$ is the hinge between the two, promising that the
coarse-grid correction handles a vector with a defect measured against
$\norm{A_\ell v}$---on the same spectral scale the smoother uses.

A disproof accordingly needs exactly one legal instance of the axioms on
which the V-cycle fails, and Example~1 is engineered to be an extremal
instance rather than a representative one.  Its smoother is the
\emph{best} one the axioms admit ($\Lambda_\ell=\spec(A_\ell)$, hence
$C_R=1$): it annihilates the entire upper spectrum in a single sweep and is,
as $(R)$ permits, nearly inert on a near-kernel component---the component
which, by the design of the theory, is the coarse grid's responsibility, and
which $(A2)$ certifies the coarse space represents (in Example 1, exactly:
the bad vector lies in the range of $P$).  No smoothing hypothesis of this
type could forbid the configuration, because $(R)$ ties $\Lambda_\ell$ to
$\spec(A_\ell)$ and never to $\lambda_{\min}(A_\ell)$.  That the examples do
not work by quietly disabling the method can be checked within the examples
themselves: on the same matrices, with the same smoother and the same $m$,
the two-grid method contracts and the W-cycle contracts uniformly
(Proposition~\ref{prop:w-dichotomy})---only the V-recursion fails.  Indeed,
under $(G3)$ and $(R)$ alone the two-grid operator always satisfies
$\spec(E^{TG}_\ell)\le\norm{K_\ell}^{2m}_{A_\ell}<1$
(Remark~\ref{rem:two-grid}): divergence is intrinsically a phenomenon of
three or more levels, which is why it sets in exactly at depth two.
Finally, nothing infinite is involved: for each fixed $m$ the three-level
counterexample is a single finite hierarchy of dimensions $1,2,4$, and the
family is infinite only across the smoothing count, as it must be---its
constants grow like $C^{2}_{A2}=4m$, and Corollary~\ref{cor:threshold} shows
that growth of this order is unavoidable for any family defeating every
$m$.

\subsection{Related work}\label{ssec:related}

\emph{Variational theory.}  For nested conforming discretisations with exact
Galerkin coarsening, uniform V-cycle convergence goes
back to Braess and Hackbusch \cite{braess-hackbusch} (full regularity) and
Bramble and Pasciak \cite{bramble-pasciak,bramble-pasciak-93} (partial
regularity); the regularity-free framework of
Bramble, Pasciak, Wang and Xu \cite{bpwx} gives level-dependent V-cycle
and uniform W-cycle and variable V-cycle bounds, and Brenner
\cite{brenner-vcycle} removed full regularity for conforming finite
elements.  The subspace-correction identity
of Xu and Zikatanov \cite{xu-zikatanov} implies that in this setting the
symmetric V-cycle is unconditionally an energy contraction on every fixed
hierarchy, so counterexamples of the type constructed here cannot exist
there: at most the rate may approach $1$, as in the two-grid-versus-V-cycle
studies of Napov and Notay \cite{napov-notay-vcycle,napov-notay-order},
whose piecewise-constant Galerkin example reaches factors $0.997$ at nine
levels but, being variational, can never cross $1$.

\emph{Cell-centred multigrid.}  Multigrid for cell-centred finite differences
with the natural transfers was developed by Wesseling and Khalil
\cite{wesseling-cc,khalil-wesseling} on the basis of the transfer-order rule
$m_p+m_r>2m$ of \cite{hemker}, which the pair $(P,\half P^T)$ meets with
equality---a first warning sign.  The relation \eqref{eq:g3-intro} appears
explicitly in Ewing and Shen \cite{ewing-shen}, in the form
$2A_{k-1}(u,v)=A_k(I_ku,I_kv)$, together with the remark that ``one step of
symmetric linear smoothing in our V-cycle multigrid scheme may fail to be a
contraction''.  The convergence framework built on \eqref{eq:g3-intro} is due
to Bramble, Ewing, Pasciak and Shen \cite{beps}, who verify the hypotheses of
the nonnested multigrid theory of Bramble, Pasciak and Xu \cite{bpx} with
prolongation energy constant $C^{*}=2$ (sharp) and approximation exponent
$\alpha=\half$, and conclude uniform convergence of the W-cycle and uniform
preconditioning bounds for the variable V-cycle.  A self-contained textbook
development of this theory is given in \cite[Chapters~6 and~8]{sww}, whose
labels $(G3)$ and $(A2)$ we follow: there the W-cycle theorem holds for
every $m\ge1$ \cite[Theorem~6.9.15]{sww}, the property $(A2)_{1/2}$ is
proved for the constant-coefficient cell-centred hierarchy
\cite[Theorem~8.6.8]{sww}, and the V-cycle question is posed as open
\cite[Remarks~6.9.16 and~8.6.10]{sww}.  The fixed V-cycle is not
covered: the V-cycle branch of \cite{bpx} requires $C^{*}\le1$ (their
condition (A.2)), the W-cycle branch carries the explicit threshold
$C^{*}<2$ (their (A.5)), and \cite{bpx} warns that beyond (A.2) the cycle
``may no longer be a reducer'', exhibiting for artificially rescaled forms a
W-cycle operator with negative eigenvalues---but no divergent cycle.  As Kwak
put it, for the natural injection ``no conclusion can be drawn for the
V-cycle using standard multigrid theory'' \cite{kwak}.  Positive V-cycle
results for cell-centred discretisations therefore modify the transfers to
restore an energy-non-expansive prolongation: Kwak \cite{kwak} and Kwak and
Lee \cite{kwak-lee,kwak-lee-tri} construct weighted and coefficient-dependent
prolongations with $C^{*}\le1$ (obtaining level-dependent V-cycle bounds and
uniform preconditioners), and Mohr and Wienands \cite{mohr-wienands} survey
the transfer-operator landscape.  Numerical evidence that the \emph{natural}
V-cycle degrades with the number of levels, and can diverge for discontinuous
coefficients, was reported already in \cite{kwak,ewing-shen}; what those
observations left open is whether the published axioms themselves exclude
such behaviour.  Our results show they do not: divergence occurs while
$(G3)$ holds exactly and $(A2)_{1/2}$ holds with verified constants.

\emph{Aggregation and over-correction.}  Piecewise-constant prolongation is
the basic ingredient of aggregation-based algebraic multigrid, and it is well
documented there that plain aggregation yields acceptable two-grid but
degraded V-cycle convergence
\cite{napov-notay-vcycle,muresan-notay,hu-vassilevski-xu}, which is why
aggregation AMG is run with over-correction, smoothed aggregates, AMLI
cycles, or Krylov-accelerated (K-)cycles
\cite{notay-vassilevski,blaheta,braess-amg,vanek-mandel-brezina,axelsson-vassilevski,%
notay-aggregation}.  In that literature the over-correction
factor is kept strictly below $2$ (Braess uses $1.8$, with $\approx2$
reported as optimal but avoided ``to prevent overshooting''
\cite{braess-amg,bolten-etal}), and the V-cycle analyses concern the Galerkin
normalisation, where divergence is impossible.  A classical relative is the
one-dimensional AMG example of Brandt and of Ruge and St\"uben
\cite{tos,ruge-stuben} in which each Galerkin coarse operator drifts from the
rediscretisation by a factor of $2$ and the V-cycle factor degrades like
$1-2^{-L}$---again inside the variational frame, where crossing $1$ is
impossible \cite[Cor.~A.2.1]{tos}.  Condition \eqref{eq:g3-intro}
is precisely aggregation with the exact over-correction factor $2$
(Lemma~\ref{lem:structure}), and our results show that there the V-cycle does
not merely degrade; it diverges, even though two-grid and W-cycle methods
contract and the weak approximation property holds with modest constants.
This sharpens, in the cell-centred normalisation, the general warnings that
a uniform weak approximation property alone does not control multilevel
convergence \cite{brannick-etal} and that leaving the Galerkin setting voids
the variational safety net \cite{falgout-schroder}; published examples of
genuinely divergent V-cycles had so far arisen only in nonsymmetric
space-time (MGRIT) settings \cite{hessenthaler-etal}.

\subsection{Outline}

Section~\ref{sec:framework} fixes the framework, derives the structural
consequences of $(G3)$, and records the positive two-grid and W-cycle results
(proved in Appendix~\ref{app:wcycle} for completeness).
Sections~\ref{sec:ex2} and~\ref{sec:ex3} contain the two counterexamples.
Section~\ref{sec:numerics} reports independent numerical verification of
every claim, including a numerically optimised lower bound for the true
$(A2)_{1/2}$ constants of Example~2 and a robustness study with respect to the
Richardson parameter and the smoothing count.  Section~\ref{sec:discussion}
discusses consequences and open problems.

\section{The axiomatic framework}\label{sec:framework}

On level $\ell\in\{0,\dots,L\}$ let $A_\ell\in\R^{n_\ell\times n_\ell}$ be
symmetric positive definite, let $P_{\ell-1}\in\R^{n_\ell\times n_{\ell-1}}$
have full column rank, and let $R_{\ell-1}\in\R^{n_{\ell-1}\times n_\ell}$.
We write $\ip{u}{v}_{A_\ell}=u^{T}A_\ell v$ and
$\norm{v}_{A_\ell}^{2}=v^{T}A_\ell v$, and $\spec(\cdot)$ for the spectral
radius.

\begin{definition}[Imbalanced Galerkin condition $(G3)$]\label{def:g3}
There is a fixed $r\in(0,1)$ such that for every $\ell\ge1$,
\begin{equation}\label{eq:g3}
  R_{\ell-1}=r\,P_{\ell-1}^{T},
  \qquad
  R_{\ell-1}A_\ell P_{\ell-1}=2\,A_{\ell-1}.
\end{equation}
\end{definition}

In the cell-centred setting of Section~\ref{sec:ex3} one has $r=\half$ and
$P^{T}_{\ell-1}A_\ell P_{\ell-1}=4A_{\ell-1}$; the abstract family of
Section~\ref{sec:ex2} also uses $r=\half$.

\begin{definition}[Richardson smoothing $(R)$]\label{def:richardson}
Each level $\ell\ge1$ carries the smoother
\[
  K_\ell=I_\ell-\Lambda_\ell^{-1}A_\ell,
  \qquad
  \spec(A_\ell)\le\Lambda_\ell\le C_R\,\spec(A_\ell),
\]
with a level-independent constant $C_R\ge1$.  A sequence
$(\Lambda_\ell)_{\ell\ge1}$ satisfying only the lower bound
$\Lambda_\ell\ge\spec(A_\ell)$ is called \emph{admissible}.
\end{definition}

Admissibility is the one-sided half of $(R)$: it makes $K_\ell$ a positive
semidefinite $A_\ell$-contraction but places no upper bound on
$\Lambda_\ell$.  The positive results below use the constant $C_R$; the
divergence result for the finite-volume hierarchy
(Theorem~\ref{thm:ode-divergence}) uses admissibility alone.

Since $0\prec A_\ell\preceq\Lambda_\ell I$, the smoother $K_\ell$ is
symmetric positive semidefinite in the $A_\ell$-inner product and
$\norm{K_\ell}_{A_\ell}=1-\lambda_{\min}(A_\ell)/\Lambda_\ell<1$.

With $m$ pre- and post-smoothing steps and $p\ge1$ recursive coarse calls, the
error propagation operators of the symmetric multigrid cycle are defined by
$E_0=0$ and
\begin{equation}\label{eq:mg-recursion}
  E_\ell
  =K_\ell^{m}\Bigl(I_\ell-P_{\ell-1}\bigl(I_{\ell-1}-E_{\ell-1}^{\,p}\bigr)
   \Pi_{\ell-1}\Bigr)K_\ell^{m},
  \qquad
  \Pi_{\ell-1}=A_{\ell-1}^{-1}R_{\ell-1}A_\ell ,
\end{equation}
so that $p=1$ is the V-cycle and $p=2$ the W-cycle
(Algorithm~\ref{alg:cycle}): $E_\ell$ is the error propagation
operator of one call $u\leftarrow\mathrm{MG}(\ell,g_\ell,u)$, and the exact
coarse solve on level $0$ is encoded by $E_0=0$.  We write
$\Pt_\ell:=P_{\ell-1}\Pi_{\ell-1}$ for the coarse-grid correction operator.

\begin{definition}[Weak approximation property $(A2)_\alpha$]\label{def:a2}
Let $\alpha\in(0,1]$.  The hierarchy satisfies $(A2)_\alpha$ if there is a
level-independent constant $C_{A2}$ such that for all $\ell\ge1$ and all
$v\in\R^{n_\ell}$,
\begin{equation}\label{eq:a2}
  \bigl|\ip{(I_\ell-\Pt_\ell)v}{v}_{A_\ell}\bigr|
  \;\le\;
  \frac{C_{A2}^{2\alpha}}{\spec(A_\ell)^{\alpha}}\,
  \norm{A_\ell v}^{2\alpha}\,
  \norm{v}_{A_\ell}^{2(1-\alpha)} .
\end{equation}
\end{definition}

Condition \eqref{eq:a2} is the natural transcription, to the intergrid
operator of \eqref{eq:mg-recursion}, of the scale of regularity-free
conditions of Bramble and Pasciak \cite{bramble-book,bramble-pasciak}; the
case $\alpha=\half$ is the form established for cell-centred discretisations
in \cite{beps}.  Definitions \ref{def:g3} and \ref{def:a2} follow, in
labels, normalisation, and the two-sided (absolute-value) form, the
textbook formulation of \cite[Assumption~6.9.1 and \S6.9]{sww}; the
absolute value makes our counterexamples stronger, as they satisfy the most
demanding variant of the hypothesis.  In the normalisation of the nonnested framework of
\cite{beps,bpx}, where the level forms carry the physical scaling
($a_\ell(v,w)=h_\ell\,v^{T}A_\ell w$ in one dimension), condition $(G3)$ with
$r=\half$ is exactly the borderline prolongation-energy identity
$a_\ell(Pw,Pw)=2\,a_{\ell-1}(w,w)$, i.e.\ the sharp case $C^{*}=2$ of their
condition (A.2)/(A.5); the normalisation-free formulation is
Lemma~\ref{lem:structure}(a) below.

\subsection{Structure of the coarse-grid correction under \texorpdfstring{$(G3)$}{(G3)}}

\begin{lemma}[Reflection structure]\label{lem:structure}
Assume $(G3)$ and let
\[
  \PiG_\ell
  =P_{\ell-1}\bigl(P_{\ell-1}^{T}A_\ell P_{\ell-1}\bigr)^{-1}
   P_{\ell-1}^{T}A_\ell
\]
denote the $A_\ell$-orthogonal projection onto
$\operatorname{range}(P_{\ell-1})$.  Then, for every $\ell\ge1$:
\begin{enumerate}[label=(\alph*),itemsep=1pt,topsep=2pt]
\item $\Pt_\ell=2\,\PiG_\ell$; in particular $\Pt_\ell^{2}=2\Pt_\ell$, and
  $I-\Pt_\ell=I-2\PiG_\ell$ is $A_\ell$-self-adjoint with
  $(I-\Pt_\ell)^{2}=I$: the two-level coarse-grid correction is an
  $A_\ell$-isometric reflection.
\item For all $x\in\R^{n_{\ell-1}}$, $y\in\R^{n_\ell}$:
  $\ip{P_{\ell-1}x}{y}_{A_\ell}=\tfrac1r\ip{x}{\Pi_{\ell-1}y}_{A_{\ell-1}}$.
\item For all $y\in\R^{n_\ell}$:
  $\norm{\Pi_{\ell-1}y}_{A_{\ell-1}}^{2}
   =2r\,\norm{\PiG_\ell y}_{A_\ell}^{2}\le 2r\norm{y}_{A_\ell}^{2}$.
\item For every $v\in\R^{n_\ell}$,
  \[
    \ip{(I-\Pt_\ell)v}{v}_{A_\ell}
    =\norm{(I-\PiG_\ell)v}_{A_\ell}^{2}
     -\norm{\PiG_\ell v}_{A_\ell}^{2}.
  \]
  Consequently,
  \[
    \bigl|\ip{(I-\Pt_\ell)v}{v}_{A_\ell}\bigr|
    \le \norm{v}^2_{A_\ell}.
  \]
  Condition $(A2)_\alpha$ states that the $A_\ell$-orthogonal energy splitting
  between the coarse space and its complement is even, up to a defect
  controlled by $\norm{A_\ell v}$.
\end{enumerate}
\end{lemma}

\begin{proof}
By $(G3)$, $A_{\ell-1}=\half R_{\ell-1}A_\ell P_{\ell-1}
=\frac r2P^{T}_{\ell-1}A_\ell P_{\ell-1}$, so
\[
  \Pt_\ell=P_{\ell-1}A_{\ell-1}^{-1}R_{\ell-1}A_\ell
  =P_{\ell-1}\,\frac2r\bigl(P^{T}A_\ell P\bigr)^{-1}\,r\,P^{T}A_\ell
  =2\,\PiG_\ell ,
\]
which gives (a) since $\PiG_\ell$ is an $A_\ell$-orthogonal projection.  For
(b),
\[
  \ip{P x}{y}_{A_\ell}=x^{T}P^{T}A_\ell y=\tfrac1r\,x^{T}R A_\ell y
  =\tfrac1r\,x^{T}A_{\ell-1}\Pi_{\ell-1}y .
\]
For (c), write
$\Pi_{\ell-1}=\frac2r\bigl(P^TA_\ell P\bigr)^{-1}rP^TA_\ell=2W$, where
$PW=\PiG_\ell$; then
\[
  \norm{\Pi y}^{2}_{A_{\ell-1}}=4\,W^{T}\!A_{\ell-1}W
  =4\,W^{T}\tfrac r2(P^TA_\ell P)W=2r\norm{PW}_{A_\ell}^{2}.
\]
Part (d) follows from (a) and the Pythagoras identity
$\norm{v}^{2}=\norm{\PiG v}^{2}+\norm{(I-\PiG)v}^{2}$.
\end{proof}

Lemma \ref{lem:structure}(a) isolates the mechanism exploited in this paper:
under $(G3)$ an exact coarse solve reflects the error across the
$A_\ell$-orthogonal complement of the coarse space, preserving its energy
exactly.  Two-grid and W-cycle methods still contract for $m$ large because
the smoother contracts and because squaring the coarse-level error operator
neutralises the sign of the reflection.  The V-cycle has no such protection.

\subsection{The positive results: two-grid and W-cycle}

For orientation, and to make the dichotomy with the V-cycle precise, we record
the standard positive result in our setting; a short, self-contained proof is
given in Appendix~\ref{app:wcycle}.  Set
\begin{equation}\label{eq:epsm}
  \varepsilon(m):=C_{A2}^{2\alpha}\Bigl(\frac{C_R}{2m+1}\Bigr)^{\!\alpha}.
\end{equation}

\begin{theorem}[Two-grid and W-cycle convergence under
$(G3)+(R)+(A2)_\alpha$]\label{thm:wcycle}
Assume $(G3)$, $(R)$ and $(A2)_\alpha$ for some $\alpha\in(0,1]$.  Then for
every $\ell\ge1$ the two-grid operator
$E^{TG}_\ell=K_\ell^m(I-\Pt_\ell)K_\ell^m$ satisfies
$\spec(E^{TG}_\ell)\le\varepsilon(m)$.  Moreover, if $m$ is large enough that
$\varepsilon(m)\le\frac18$, the symmetric W-cycle satisfies
\[
  \sup_{\ell\ge1}\;\spec\bigl(E^{W}_\ell\bigr)\;\le\;
  \delta:=\frac{1-\sqrt{1-8\varepsilon(m)}}{4}\;\le\;\frac14 .
\]
\end{theorem}

Theorem~\ref{thm:wcycle} requires $m\gtrsim C_R\,(8C_{A2}^{2\alpha})^{1/\alpha}$;
we include it only to make the paper self-contained and the dichotomy with
the V-cycle precise.  Sharper W-cycle statements exist: in exactly the
present axiomatic setting, \cite[Theorem~6.9.15]{sww} gives
$|\ip{E^{W}_\ell v}{v}_{A_\ell}|\le\frac{M}{M+m^{\alpha}}\norm{v}^{2}_{A_\ell}$
for \emph{every} $m\ge1$, with $M=\max\{M_0,m^{\alpha}\}$ for a constant
$M_0=M_0(C_R,C_{A2},\alpha)$; the nonnested theory
of \cite{bpx} (applied to cell-centred hierarchies in \cite{beps}) covers the
borderline energy constant $C^{*}=2$ with few smoothing steps; and Notay
\cite{notay-perturbed} derives uniform W-cycle bounds from a two-grid factor
below $\half$ without any Galerkin relation---while observing that for the
V-cycle the same machinery ``fails to deliver bounds independent of the
number of levels''.  The examples below satisfy these hypotheses with
explicit constants and nevertheless defeat the V-cycle: on the family of
Section~\ref{sec:ex2} the W-cycle contracts uniformly \emph{at the
very smoothing count for which the V-cycle diverges}
(Proposition~\ref{prop:w-dichotomy}); on the finite-volume hierarchy of
Section~\ref{sec:ex3} the same contrast is observed numerically
(Section~\ref{sec:numerics}).

\begin{remark}[Two-grid methods cannot diverge]\label{rem:two-grid}
Theorem~\ref{thm:wcycle} invokes $(A2)_\alpha$ only through
$\varepsilon(m)$.  One unconditional fact deserves emphasis: under $(G3)$
and $(R)$ alone,
\[
  \spec\bigl(E^{TG}_\ell\bigr)
  \le\bigl\|K_\ell^{m}(I-\Pt_\ell)K_\ell^{m}\bigr\|_{A_\ell}
  \le\norm{K_\ell}^{2m}_{A_\ell}\;<\;1,
\]
because $I-\Pt_\ell$ is an $A_\ell$-isometry by
Lemma~\ref{lem:structure}(a) and
$\norm{K_\ell}_{A_\ell}=1-\lambda_{\min}(A_\ell)/\Lambda_\ell<1$.  A
two-grid method satisfying the imbalanced Galerkin condition can therefore
never diverge, however weak the smoother: divergence requires at least one
\emph{inexact} coarse solve whose sign-carrying error the factor $2$ of
$(G3)$ can amplify.
\end{remark}

\subsection{The V-cycle questions}

The following two questions, both asking for bounds independent of the
number of levels, calibrate exactly what the counterexamples show.

\begin{question}[fixed smoothing]\label{q:fixed}
Does $(G3)+(R)+(A2)_\alpha$ imply, for each fixed $m\ge1$, a bound
$\sup_\ell\spec(E^{V}_\ell)\le\delta(C_R,C_{A2},\alpha,m)<1$?
\end{question}

\begin{question}[threshold smoothing]\label{q:threshold}
Is there a function $m_0(C_R,C_{A2},\alpha)$ such that
$(G3)+(R)+(A2)_\alpha$ and $m\ge m_0$ imply
$\sup_\ell\spec(E^{V}_\ell)\le\delta<1$?
\end{question}

Theorems \ref{thm:main-ex2} and~\ref{thm:main-ex3} answer
Question~\ref{q:fixed} negatively for every $m$, with divergence rather than
mere non-uniformity.  Question~\ref{q:threshold} remains open, but
Theorem~\ref{thm:main-ex2} shows any admissible threshold obeys
$m_0\ge\frac{\ln3}{2}(C_{A2}^{2}-1)$; for comparison, the W-cycle threshold in
Theorem~\ref{thm:wcycle} is $m_0^{W}\sim C_R(8C_{A2}^{2\alpha})^{1/\alpha}$,
and \cite[Theorem~6.9.15]{sww} removes the W-cycle threshold altogether.

\section{Example 1: an algebraic family divergent at every depth, for
every fixed \texorpdfstring{$m$}{m}}\label{sec:ex2}

Fix $m\in\N$ and set $\tau=1/(4m)$.  Let $n_0=1$ and $n_\ell=2^{\ell}$, and
define
\[
  A_0=\bigl(\tfrac{\tau}{4}\bigr),
  \qquad
  A_\ell=\diag(\tau,1,\dots,1)\in\R^{n_\ell\times n_\ell}\quad(\ell\ge1),
\]
\[
  P_0=\begin{pmatrix}1\\0\end{pmatrix},
  \qquad
  P_\ell=\begin{pmatrix}2I_{n_\ell}\\0\end{pmatrix}
  \in\R^{n_{\ell+1}\times n_\ell},
  \qquad
  R_\ell=\tfrac12P_\ell^{T}\quad(\ell\ge1),
\]
with $\Lambda_\ell=1=\spec(A_\ell)$.  Every finite truncation is a legal SPD
hierarchy; the depth-$2$ truncation, of dimensions $1,2,4$, is already the
three-level counterexample of Theorem~\ref{thm:main-ex2}.  The smoother
annihilates every coordinate except the first, ``algebraically smooth'' one,
on which it acts as multiplication by $1-\tau$.

\begin{proposition}[Uniform hypotheses]\label{prop:ex2-hyp}
For every fixed $m$ and every depth $L$, the hierarchy satisfies $(G3)$ with
$r=\half$, Richardson admissibility with $C_R=1$, and $(A2)_\alpha$ for every
$\alpha\in(0,1]$ with the level-independent sharp constant
$C_{A2}^{2}=4m$, attained at $v=e_1$ simultaneously for every $\alpha$.
\end{proposition}

\begin{proof}
For $\ell\ge2$ write $v=(x,y)\in\R^{n_{\ell-1}}\times\R^{n_{\ell-1}}$ and
\[
  A_\ell=\begin{pmatrix}A_{\ell-1}&0\\0&I\end{pmatrix},
  \qquad
  P_{\ell-1}=\begin{pmatrix}2I\\0\end{pmatrix},
  \qquad
  R_{\ell-1}=(I,\,0).
\]
Then $R_{\ell-1}A_\ell P_{\ell-1}=2A_{\ell-1}$, which is $(G3)$, and
$\spec(A_\ell)=1$.  The intergrid operators are
\[
  \Pi_{\ell-1}=(I,\,0),
  \qquad
  \Pt_\ell=\begin{pmatrix}2I&0\\0&0\end{pmatrix},
  \qquad
  I-\Pt_\ell=\begin{pmatrix}-I&0\\0&I\end{pmatrix}.
\]
With
\[
  a=-x^{T}A_{\ell-1}x+\norm{y}^{2},
  \qquad
  b=\norm{A_{\ell-1}x}^{2}+\norm{y}^{2},
  \qquad
  c=x^{T}A_{\ell-1}x+\norm{y}^{2},
\]
one has $|a|\le c$, and since every eigenvalue of $A_{\ell-1}$ lies in
$[\tau,1]$, also $c\le\tau^{-1}b$.  Hence, using $\spec(A_\ell)=1$,
\[
  |a|\le c=c^{1-\alpha}c^{\alpha}
  \le\bigl(\tau^{-1}\bigr)^{\alpha}\,b^{\alpha}\,c^{1-\alpha}
  =(4m)^{\alpha}\,\norm{A_\ell v}^{2\alpha}\norm{v}_{A_\ell}^{2(1-\alpha)},
\]
which is \eqref{eq:a2} with $C_{A2}^{2}=\tau^{-1}=4m$.  At $v=e_1$ one has
$|a|=\tau$, $b=\tau^{2}$, $c=\tau$, and the right-hand side of \eqref{eq:a2}
equals $\tau^{-\alpha}\tau^{2\alpha}\tau^{1-\alpha}=\tau$: equality holds for
every $\alpha$.  At $\ell=1$,
\[
  R_0A_1P_0=\tfrac12\,(1,0)\diag(\tau,1)\begin{pmatrix}1\\0\end{pmatrix}
  =\frac{\tau}{2}=2A_0,
\]
and $\Pt_1=P_0A_0^{-1}R_0A_1$ gives $I-\Pt_1=\diag(-1,1)$, so the same
bounds hold verbatim with $a=-\tau v_1^{2}+v_2^{2}$,
$b=\tau^{2}v_1^{2}+v_2^{2}$, $c=\tau v_1^{2}+v_2^{2}$.
\end{proof}

\begin{theorem}[Closed form for the unstable eigenvalue]\label{thm:ex2-closed}
Let $\theta=\theta_m=(1-\frac1{4m})^{2m}$ and let
$E^{V}_\ell$ denote the symmetric $V(m,m)$-cycle operators on the depth-$L$
truncation.  Then for $1\le\ell\le L$,
\[
  E^{V}_\ell=\diag(-\alpha_\ell,\,0,\dots,0),
  \qquad
  \alpha_1=\theta,
  \quad
  \alpha_\ell=\theta(1+2\alpha_{\ell-1})\quad(\ell\ge2),
\]
so that
\[
  \spec\bigl(E^{V}_\ell\bigr)=\alpha_\ell
  =\theta\,\frac{(2\theta)^{\ell}-1}{2\theta-1}\;\longrightarrow\;\infty
  \qquad(\ell\to\infty),
\]
geometrically with ratio $2\theta_m\in(1,2)$; in particular
$\spec(E^{V}_2)=\theta_m(1+2\theta_m)>1$ for every $m\in\N$.
\end{theorem}

\begin{proof}
At level one, $\Pi_0=A_0^{-1}R_0A_1=\frac4\tau\cdot\tfrac12(\tau,0)=(2,0)$,
so $I-P_0\Pi_0=\diag(-1,1)$ and, since $K_1^{m}=\diag((1-\tau)^{m},0)$,
\[
  E^{V}_1=K_1^{m}\bigl(I-P_0\Pi_0\bigr)K_1^{m}=\diag(-\theta,\,0):
  \qquad \alpha_1=\theta.
\]
Let $E^{V}_{\ell-1}=\diag(-\alpha_{\ell-1},0,\dots,0)$.
Then $I-E^{V}_{\ell-1}=\diag(1+\alpha_{\ell-1},1,\dots,1)$, and with
$\Pi_{\ell-1}=(I,0)$, $P_{\ell-1}=(2I,0)^{T}$,
\[
  I-P_{\ell-1}(I-E^{V}_{\ell-1})\Pi_{\ell-1}
  =\diag\bigl(-1-2\alpha_{\ell-1},\,-1,\dots,-1,\,1,\dots,1\bigr).
\]
Since $K_\ell^{m}=\diag((1-\tau)^{m},0,\dots,0)$ annihilates every coordinate
whose $A_\ell$-eigenvalue is $1$, sandwiching leaves
$E^{V}_\ell=\diag(-\theta(1+2\alpha_{\ell-1}),0,\dots,0)$.  The recursion
$\alpha_\ell=\theta+2\theta\alpha_{\ell-1}$ sums to the stated closed form.
It remains to prove $\theta_m>\half$.  The function
$g(x)=(1-x)^{2m}-(1-2mx)$ satisfies $g(0)=0$ and
$g'(x)=2m\bigl(1-(1-x)^{2m-1}\bigr)>0$ for $x\in(0,1)$, so
$\theta_m=(1-\tau)^{2m}>1-2m\tau=\half$; hence $2\theta_m>1$ and
$\alpha_2=\theta_m(1+2\theta_m)>\half\cdot2=1$.
\end{proof}

\begin{remark}\label{rem:ex1-mechanism}
The error
$e_1$ is invisible to the smoother up to the factor $(1-\tau)^{m}$, and the
level-$\ell$ V-cycle returns it \emph{reflected and scaled}:
$E^{V}_\ell e_1=-\alpha_\ell e_1$.  At the next level the coarse correction
applies $I-P(I-E^{V}_\ell)\Pi$, and by $(G3)$ the term $PE^{V}_\ell\Pi$ acts
on $e_1$ as $2\times(-\alpha_\ell)$: the imbalance factor doubles the
reflected coarse defect, producing $-(1+2\alpha_\ell)$, while the smoother
costs only the factor $\theta_m>\half$.  For $m=1$ the three-level radius is
$\spec(E^{V}_2)=\frac{9}{16}\bigl(1+\frac98\bigr)=\frac{153}{128}
\approx1.195$; as $m\to\infty$ it increases towards
$e^{-1/2}(1+2e^{-1/2})\approx1.342$.
\end{remark}

The same one-dimensional reduction yields two further consequences.
First, the W-cycle is uniformly contractive on the entire family
\emph{at the same smoothing count} $m$:

\begin{proposition}[W/V dichotomy on the family]\label{prop:w-dichotomy}
On the depth-$L$ hierarchy of Proposition~\ref{prop:ex2-hyp}, the
symmetric $W(m,m)$-cycle error operators have the form
$E^{W}_\ell=\diag(e_\ell,0,\dots,0)$ with
\[
  e_1=-\theta_m,
  \qquad
  e_\ell=\theta_m\bigl(2e_{\ell-1}^{2}-1\bigr)\quad(\ell\ge2),
\]
and
\[
  \spec\bigl(E^{W}_\ell\bigr)=|e_\ell|\;\le\;\theta_m\;<\;e^{-1/2}\approx0.6065
  \qquad\text{for all } \ell,\,m,\,L .
\]
\end{proposition}

\begin{proof}
The recursion follows as in Theorem~\ref{thm:ex2-closed} with
$E_{\ell-1}^{2}=\diag(e_{\ell-1}^{2},0,\dots,0)$ in place of $E_{\ell-1}$, the
first diagonal entry of the middle factor being
$1-2(1-e_{\ell-1}^{2})=2e_{\ell-1}^{2}-1$.  If $|e_{\ell-1}|\le\theta_m<1$
then $|2e_{\ell-1}^{2}-1|\le1$, so $|e_\ell|\le\theta_m$; induction starts at
$|e_1|=\theta_m$.  Finally $\theta_m=(1-\frac1{4m})^{2m}$ increases to
$e^{-1/2}$.
\end{proof}

Second, running the $V(\mu,\mu)$-cycle on the hierarchy built with parameter
$m$ (so that $C_{A2}^{2}=4m$ is fixed) replaces $\theta_m$ by
$\theta(\mu)=(1-\frac1{4m})^{2\mu}$ in Theorem~\ref{thm:ex2-closed}.  The
geometric divergence persists as long as $\theta(\mu)>\half$, and uniform
contractivity over all depths requires
$\sup_L\alpha_L=\theta(\mu)/(1-2\theta(\mu))<1$, i.e.\ $\theta(\mu)<\frac13$:

\begin{corollary}[Lower bound for any V-cycle smoothing
threshold]\label{cor:threshold}
On the family of Proposition~\ref{prop:ex2-hyp} with $C_{A2}^{2}=4m$:
\begin{enumerate}[label=(\roman*),itemsep=1pt,topsep=2pt]
\item the $V(\mu,\mu)$-cycle spectral radii diverge geometrically in the
depth whenever $\mu\le\frac{\ln2}{2}\,(4m-1)$;
\item $\sup_L\spec(E^{V}_L)>1$ (no uniform contraction) whenever
$\mu\le\frac{\ln3}{2}\,(4m-1)$.
\end{enumerate}
Consequently, any threshold function $m_0$ answering
Question~\ref{q:threshold} affirmatively must satisfy
$m_0(C_R{=}1,C_{A2},\alpha)>\frac{\ln3}{2}\bigl(C_{A2}^{2}-1\bigr)$ for every
$\alpha\in(0,1]$.
\end{corollary}

\begin{proof}
With $s=4m$, $\theta(\mu)=(1-\frac1s)^{2\mu}\ge e^{-2\mu/(s-1)}$ because
$\ln(1-\frac1s)\ge-\frac1{s-1}$.  Thus $\theta(\mu)\ge\half$ when
$2\mu\le(s-1)\ln2$, and $\theta(\mu)\ge\frac13$ when $2\mu\le(s-1)\ln3$; in
the latter case $\sup_L\alpha_L=\theta/(1-2\theta)\ge1$ (interpreted as
$+\infty$ for $\theta\ge\half$).
\end{proof}

Figure~\ref{fig:ex2} illustrates the geometric growth of $\alpha_L$ with
the depth and the three-level factor $\alpha_2>1$ across the smoothing
counts.

\begin{figure}[t]
\centering
\includegraphics[width=.47\linewidth]{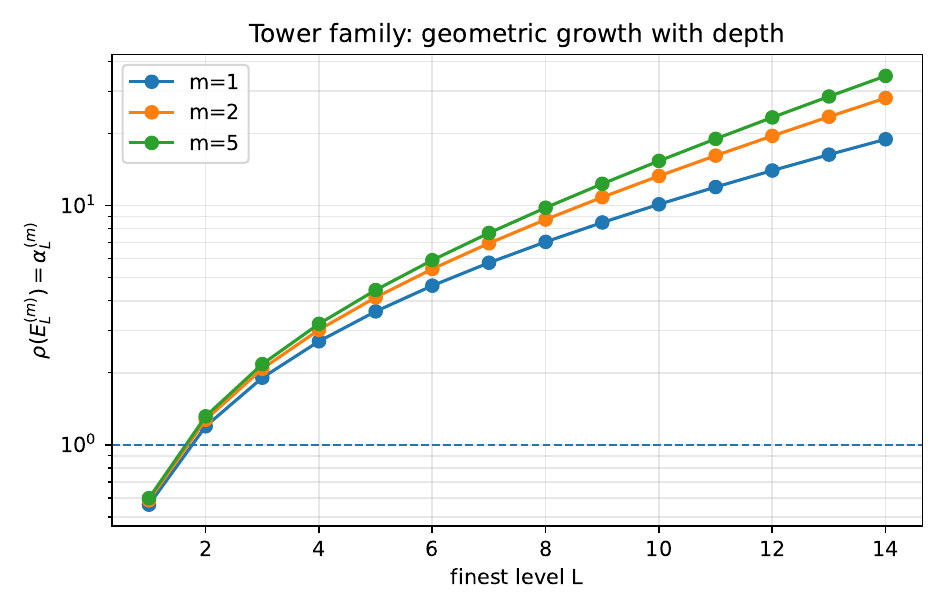}\hfill
\includegraphics[width=.47\linewidth]{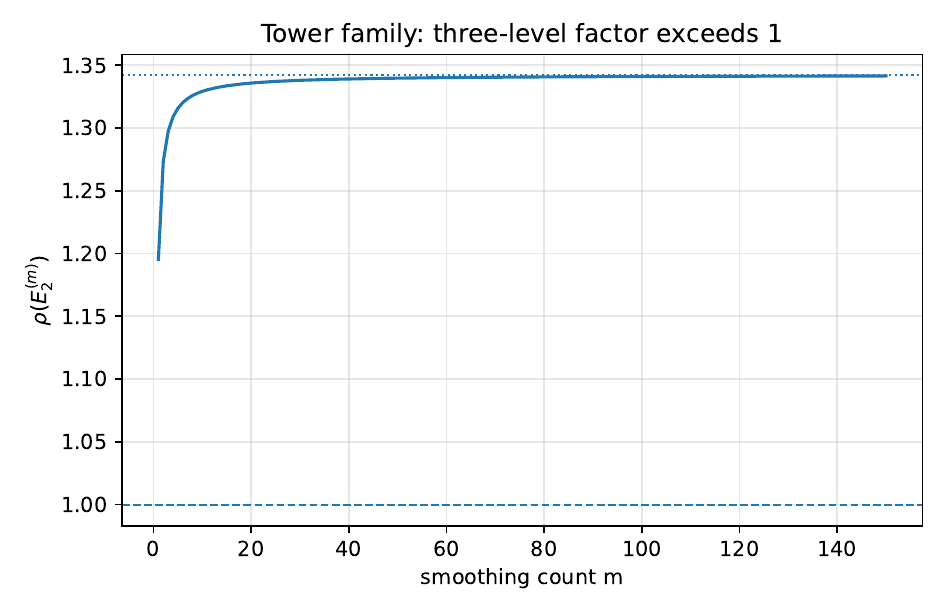}
\caption{Example 1.  Left: for fixed $m$, the unstable eigenvalue
$\alpha_L$ grows geometrically with the depth $L$ (markers: assembled
operators; lines: closed form of Theorem~\ref{thm:ex2-closed}).  Right: the
three-level factor $\alpha_2=\theta_m(1+2\theta_m)$ stays above $1$ for
every $m$.}
\label{fig:ex2}
\end{figure}

\section{Example 2: a cell-centred finite-volume hierarchy with harmonic
interface averaging}\label{sec:ex3}

Example 1 is algebraic.  This section shows that nothing about it
is synthetic: the identical failure occurs in the standard cell-centred
finite-volume discretisation \cite[Chapter~8]{sww} of a one-dimensional
interface problem, with piecewise-constant transfers and the standard
(Samarskii/harmonic) treatment of the coefficient jump.

\subsection{The hierarchy}\label{ssec:ex3-setting}

Fix $\kappa>1$ and consider
\begin{equation}\label{eq:ode}
  -\bigl(D(x)u'(x)\bigr)'=f(x)\quad\text{in }(-1,1),
  \qquad u(-1)=u(1)=0,
  \qquad
  D=\begin{cases}1,&x<0,\\ \kappa,&x>0.\end{cases}
\end{equation}
On level $\ell\ge1$ use the uniform mesh with $n_\ell=2^{\ell}$ cells
$\mathcal C_i=(x_{i-1/2},x_{i+1/2})$ of width $h_\ell=2/n_\ell$ and centres
$x_i$; the interface $x=0$ is a cell edge on every level.  Write
$D_i\in\{1,\kappa\}$ for the (constant) coefficient on $\mathcal C_i$.
Interior edges carry the harmonic average
\begin{equation}\label{eq:harmonic}
  D_{i+1/2}=\frac{2D_iD_{i+1}}{D_i+D_{i+1}}
  =\Bigl(\frac1{h_\ell}\int_{x_i}^{x_{i+1}}\frac{ds}{D(s)}\Bigr)^{-1},
\end{equation}
which at the interface is the Samarskii balance-method coefficient
$\eta:=2\kappa/(1+\kappa)$ \cite{samarskii,egh}; boundary edges carry the
adjacent cell value, $D_{1/2}=D_1$ and $D_{n+1/2}=D_n$.  The discrete operator is the ghost-cell
Dirichlet finite-volume matrix
\begin{equation}\label{eq:fv-matrix}
  (A_\ell u)_i
  =\frac{-D_{i+1/2}(u_{i+1}-u_i)+D_{i-1/2}(u_i-u_{i-1})}{h_\ell^{2}},
  \qquad
  u_0=-u_1,\quad u_{n_\ell+1}=-u_{n_\ell}.
\end{equation}
Coarsening aggregates adjacent cell pairs,
\[
  (P_{\ell-1}w)_{2j-1}=(P_{\ell-1}w)_{2j}=w_j,
  \qquad
  R_{\ell-1}=\tfrac12P_{\ell-1}^{T},
\]
and the coarse operator $A_{\ell-1}$ is the finite-volume matrix
\eqref{eq:fv-matrix} on the coarse mesh, whose edge coefficients are obtained
from the fine ones by every-other-edge injection.  Because the jump sits on a
dyadic edge, each coarse cell on the levels $\ell-1\ge1$ is a single
material, every coarse edge
coefficient again has the form \eqref{eq:harmonic}, and the interface edge
carries the same harmonic value $\eta$ on every level: the coarse operators
are simultaneously \emph{rediscretisations} and---as the next proposition
shows---\emph{exact} imbalanced Galerkin operators.  Level $0$ consists of a
single cell straddling the interface ($n_0=1$, $h_0=2$), whose injected edge
coefficients are $D^{(0)}_{1/2}=1$ and $D^{(0)}_{3/2}=\kappa$, so that
\eqref{eq:fv-matrix} (with both neighbours of the single cell being ghost
cells) gives
\[
  A_0=\frac{2\cdot1+2\cdot\kappa}{h_0^{2}}=\frac{1+\kappa}{2}.
\]
The cycle uses level $0$ only through the exact solve $E_0=0$, but $A_0$
enters $\Pi_0=A_0^{-1}R_0A_1$, and $(G3)$ holds at $\ell=1$ as well, as the
proof below verifies directly.

\subsection{Exact \texorpdfstring{$(G3)$}{(G3)}}

\begin{proposition}[Exact $(G3)$ for the harmonic hierarchy]\label{prop:g3-ode}
For every $\ell\ge1$,
\[
  R_{\ell-1}=\tfrac12P_{\ell-1}^{T},
  \qquad
  R_{\ell-1}A_\ell P_{\ell-1}=2A_{\ell-1},
  \qquad\text{equivalently}\quad
  P_{\ell-1}^{T}A_\ell P_{\ell-1}=4A_{\ell-1}.
\]
\end{proposition}

\begin{proof}
Write $h=h_\ell$ and let $D_j^{e}$ denote the fine edge coefficients,
$j=\frac12,\frac32,\dots$.  The matrix $P^{T}A_\ell P$ acts on coarse vectors
by first copying to fine cells, then applying $A_\ell$, then summing over
pairs.  Fix an interior coarse cell $k$, consisting of fine cells $2k-1,2k$.
Summing rows $2k-1$ and $2k$ of \eqref{eq:fv-matrix} cancels the flux through
the interior edge $2k-\frac12$ of the pair, leaving only the two outer edges
$2k-\frac32$ and $2k+\frac12$:
\begin{multline*}
  \bigl[(A_\ell Pw)_{2k-1}+(A_\ell Pw)_{2k}\bigr]\\
  =h^{-2}\Bigl[-D^{e}_{2k+1/2}\bigl((Pw)_{2k+1}-(Pw)_{2k}\bigr)
        +D^{e}_{2k-3/2}\bigl((Pw)_{2k-1}-(Pw)_{2k-2}\bigr)\Bigr] .
\end{multline*}
Since $(Pw)_{2k+1}=w_{k+1}$, $(Pw)_{2k}=(Pw)_{2k-1}=w_{k}$,
$(Pw)_{2k-2}=w_{k-1}$, this equals
\[
  \frac{-D^{e}_{2k+1/2}(w_{k+1}-w_k)+D^{e}_{2k-3/2}(w_k-w_{k-1})}{h^{2}} .
\]
The outer fine edges $2k\pm\ldots$ are exactly the coarse edges of cell $k$,
and the injected coarse coefficients are
$D^{e,\,\mathrm{coarse}}_{k+1/2}=D^{e}_{2k+1/2}$.  Multiplying by $\frac14$
converts $h^{-2}$ into $h_{\ell-1}^{-2}=(2h)^{-2}$ and yields precisely row
$k$ of the coarse finite-volume matrix.  At the left boundary, summing rows
$1$ and $2$ leaves the ghost contribution $2D^{e}_{1/2}(Pw)_1/h^{2}$ and the
outer edge $D^{e}_{5/2}$; with $(Pw)_1=w_1$ this becomes, after the factor
$\frac14$, the coarse boundary row with coarse boundary coefficient
$D^{e}_{1/2}=D_1$ (the adjacent coarse cell value, as injected).  The right
boundary is symmetric.  Finally, at $\ell=1$ the single coarse cell receives
both ghost contributions at once: summing the two rows of $A_1$ against
$P_0w=(w,w)^{T}$ cancels the interface edge $\eta$ and leaves
$(2D^{e}_{1/2}+2D^{e}_{5/2})w/h_1^{2}=2(1+\kappa)\,w=4A_0w$, by the
definition of $A_0$ in Subsection~\ref{ssec:ex3-setting}.  Hence
$P^{T}A_\ell P=4A_{\ell-1}$ for every $\ell\ge1$, and multiplying by
$R=\half P^{T}$ gives \eqref{eq:g3}.
\end{proof}

\subsection{Richardson admissibility}\label{ssec:ex3-richardson}

\begin{lemma}\label{lem:rho-bounds}
For every $\ell\ge1$,
$\spec(A_\ell)\le 4\kappa h_\ell^{-2}$.  Moreover
$\spec(A_\ell)\ge3\kappa h_\ell^{-2}$ for $\ell\ge2$, and
$\spec(A_1)\ge2\kappa h_1^{-2}$.  Consequently the choice
$\Lambda_\ell=4\kappa h_\ell^{-2}$ is admissible with
$\Lambda_\ell\le2\,\spec(A_\ell)$ for all $\ell\ge1$ and
$\Lambda_\ell\le\frac43\spec(A_\ell)$ for $\ell\ge2$; in particular $(R)$
holds with $C_R=2$.
\end{lemma}

\begin{proof}
Every Gershgorin row sum of \eqref{eq:fv-matrix} is at most
$4\max_jD^{e}_j\,h^{-2}\le4\kappa h^{-2}$, since each harmonic edge value is
bounded by $\kappa$ and the boundary rows carry the weight $2D^{e}$ at the
ghost edge.  For $\ell\ge2$ the two rightmost cells lie in the
$\kappa$-region, so the Rayleigh quotient of $e_{n}$ gives
$\spec(A_\ell)\ge(A_\ell)_{nn}=(D^{e}_{n-1/2}+2D^{e}_{n+1/2})h^{-2}
=3\kappa h^{-2}$.  For $\ell=1$ ($n=2$, $h=1$), the second diagonal entry is
$(\eta+2\kappa)h^{-2}\ge2\kappa h^{-2}$.
\end{proof}

\subsection{The weak approximation property
\texorpdfstring{$(A2)_{1/2}$}{(A2)\_1/2}}\label{ssec:ex3-a2}

This subsection proves that the harmonic hierarchy satisfies $(A2)_{1/2}$
with a level-independent constant.  For the \emph{constant-coefficient}
cell-centred hierarchy this property is established in
\cite[Theorem~8.6.8]{sww}; for variable diffusivity,
\cite[Remark~8.7.3]{sww} notes that the smoothing and approximation
properties are exactly the nontrivial missing pieces.
The proof below supplies the approximation property for mesh-aligned
piecewise-constant coefficients of arbitrary contrast with harmonic
averaging.  Its key observation is that the approximation property only ever
needs to be tested against right-hand sides of the special form $f=A_\ell v$,
whose associated continuous source is \emph{piecewise constant on the cells}.
For such sources the exact flux is piecewise linear, and the consistency
error of the two-point harmonic flux can be computed \emph{in closed form}
(Lemma~\ref{lem:flux-exact}).  All constants below are absolute (they depend
neither on $\ell$, nor on $\kappa$, beyond the normalisation
$\min D=1$); $\kappa$ enters only at the very last step, through
$\spec(A_\ell)\le4\kappa h_\ell^{-2}$.

Throughout this subsection fix a level $\ell\ge2$, write $h=h_\ell$,
$n=n_\ell$, $A=A_\ell$, and let $H=2h$ refer to the coarse level $\ell-1$,
with coarse cells $\mathcal C^{c}_J=\mathcal C_{2J-1}\cup\mathcal C_{2J}$ and
centres $x^{c}_J$.  For a grid vector $w$ let $w_h$ denote the associated
piecewise-constant function, and define the discrete energy
\[
  a_h(w,w):=h\,w^{T}\!Aw
  =\sum_{i=1}^{n-1}\frac{D_{i+1/2}}{h}(w_{i+1}-w_i)^{2}
   +\frac{2D_{1/2}}{h}\,w_1^{2}+\frac{2D_{n+1/2}}{h}\,w_n^{2},
\]
so that $\norm{w_h}_{a_h}=h^{1/2}\norm{w}_{A}$ and, for $f=Aw$,
$\norm{f_h}_{L^{2}}=h^{1/2}\norm{Aw}$.

Given $v\in\R^{n}$, set $f:=Av$, let $f_h$ be the corresponding
piecewise-constant source, and let $u\in H^{1}_0(-1,1)$ solve
$-(Du')'=f_h$ weakly.  The flux
\begin{equation}\label{eq:flux}
  F:=-Du'\in H^{1}(-1,1),
  \qquad
  F'=f_h ,
\end{equation}
is continuous across the interface and \emph{piecewise linear}, with kinks
only at cell edges (the breakpoints of $f_h$; the interface is also an edge).
By construction, $v$ is exactly the finite-volume solution of
\eqref{eq:fv-matrix} with load vector $f$ (the cell averages of $f_h$), and
the coarse vector
\[
  z:=\Pi_{\ell-1}v=A_{\ell-1}^{-1}R_{\ell-1}Av=A_{\ell-1}^{-1}\bar f,
  \qquad
  \bar f:=Rf=\text{(coarse cell averages of }f_h),
\]
is exactly the coarse finite-volume solution for the same source.  Two
identities anchor the proof.  First, by Lemma~\ref{lem:structure}(a,c) (with
$r=\half$),
\begin{equation}\label{eq:T-def}
  T:=\ip{(I-\Pt_\ell)v}{v}_{A}
  =\norm{v}_{A}^{2}-2\norm{z}_{A_{\ell-1}}^{2},
  \qquad
  \norm{z}_{A_{\ell-1}}=\norm{\PiG_\ell v}_{A}\le\norm{v}_{A}.
\end{equation}
Second, testing the two finite-volume problems with their own solutions,
\begin{equation}\label{eq:hT}
  hT=a_h(v,v)-a_H(z,z)
  =(f_h,v_h)_{L^{2}}-(f_h,z_h)_{L^{2}}
  =(f_h,\,v_h-z_h)_{L^{2}} ,
\end{equation}
where $z_h$ is the piecewise-constant function of $z$ on the coarse mesh
(here we used $(\bar f_h,z_h)_{L^2}=(f_h,z_h)_{L^2}$, valid because $z_h$ is
constant on coarse cells).

Let $I_hu=(u(x_i))_i$ and $I_Hu=(u(x^{c}_J))_J$ denote the cell-centre
interpolants, and define the discrete fluxes of a grid vector $w$ by
$G_{i+1/2}(w)=-D_{i+1/2}(w_{i+1}-w_i)/h$ at interior edges and, using the
ghost values $w_0=-w_1$, $w_{n+1}=-w_n$, by
$G_{1/2}(w)=-D_{1/2}(w_1-w_0)/h=-2D_{1/2}w_1/h$ and
$G_{n+1/2}(w)=-D_{n+1/2}(w_{n+1}-w_n)/h=2D_{n+1/2}w_n/h$ at the boundary
edges, so that $(Aw)_i=\bigl(G_{i+1/2}(w)-G_{i-1/2}(w)\bigr)/h$ for all $i$;
the same definitions apply on the coarse mesh with $H$ in place of $h$.  The
\emph{flux consistency errors} of the interpolated exact solution are
\[
  \delta_e:=G_e(I_hu)-F(x_e)
  \quad\text{(fine edges $e$)},
  \qquad
  \delta^{c}_e:=G_e(I_Hu)-F(x_e)
  \quad\text{(coarse edges $e$)} .
\]

The proof now runs in three steps.  First, Lemma~\ref{lem:flux-exact}
computes every flux-consistency error \emph{in closed form}---this is where
the piecewise linearity of $F$ and the harmonic mean enter.  Second,
Lemma~\ref{lem:sbp} converts the flux errors into a residual representation
with a Cauchy--Schwarz master estimate weighted compatibly with the energy,
from which Lemma~\ref{lem:consequences} extracts the four estimates actually
used.  Third, Theorem~\ref{thm:a2-ode} splits $hT=(f_h,v_h-z_h)$ along the
chain $v_h\to I_hu\to u\to I_Hu\to z_h$ and bounds each link.

\begin{lemma}[Exact flux-error formulas]\label{lem:flux-exact}
Let $e=i+\frac12$ be an interior fine edge, with adjacent cells of constant
coefficients $D_L=D_i$, $D_R=D_{i+1}$ and edge value $D_e$ as in
\eqref{eq:harmonic}.  Then
\[
  \delta_e=\frac{D_e\,h}{8}\Bigl(\frac{f_{i+1}}{D_R}-\frac{f_i}{D_L}\Bigr).
\]
At the boundary edges, $\delta_{1/2}=\frac{h}{4}f_1$ and
$\delta_{n+1/2}=-\frac{h}{4}f_n$.  On the coarse mesh the same formulas hold
with $H$ in place of $h$ and with the \emph{fine} source values of the two
fine cells adjacent to the coarse edge; e.g.\ for an interior coarse edge $e$
between coarse cells $J,J+1$,
\[
  \delta^{c}_e=\frac{D_e\,H}{8}
  \Bigl(\frac{f_{2J+1}}{D_{2J+1}}-\frac{f_{2J}}{D_{2J}}\Bigr).
\]
\end{lemma}

\begin{proof}
Let $e$ be an interior fine edge with interval $\omega_e=(x_i,x_{i+1})$.
Since $u'=-F/D$,
\[
  G_e(I_hu)=-\frac{D_e}{h}\bigl(u(x_{i+1})-u(x_i)\bigr)
  =\frac{D_e}{h}\int_{\omega_e}\frac{F(s)}{D(s)}\,ds
  =\int_{\omega_e}w_e(s)F(s)\,ds,
\]
where $w_e=(D_e/h)/D\ge0$ satisfies $\int_{\omega_e}w_e=1$ by
\eqref{eq:harmonic}.  Hence
$\delta_e=\int_{\omega_e}w_e(s)\bigl(F(s)-F(x_e)\bigr)ds$.  On $\omega_e$ the
flux $F$ is piecewise linear with its only kink at the midpoint $x_e$ (cell
edges are the only breakpoints of $F'=f_h$, and $\omega_e$ joins two adjacent
cell centres), with slopes $f_i$ on $(x_i,x_e)$ and $f_{i+1}$ on
$(x_e,x_{i+1})$.  Also $D\equiv D_L$ on $(x_i,x_e)$ and $D\equiv D_R$ on
$(x_e,x_{i+1})$.  Therefore
\[
  \delta_e
  =\frac{D_e}{h}\Bigl[\frac{f_i}{D_L}\int_{x_i}^{x_e}(s-x_e)\,ds
   +\frac{f_{i+1}}{D_R}\int_{x_e}^{x_{i+1}}(s-x_e)\,ds\Bigr]
  =\frac{D_e}{h}\Bigl[-\frac{f_i}{D_L}+\frac{f_{i+1}}{D_R}\Bigr]
   \frac{h^{2}}{8}.
\]
At the left boundary, $u(x_1)=u(x_1)-u(-1)=-\frac1{D_1}\int_{-1}^{x_1}F$, so
that $G_{1/2}(I_hu)=-2D_1u(x_1)/h$ equals $\frac2h\int_{-1}^{x_1}F$, and
since $F$ is linear with slope $f_1$ on $(-1,x_1)$ (no interior kink),
$\delta_{1/2}=\frac2h\,f_1\int_{-1}^{x_1}(s+1)\,ds$, i.e.\
$\delta_{1/2}=\frac2h f_1\frac{(h/2)^2}2=\frac h4f_1$; the right boundary is
symmetric with a sign change.  On the
coarse mesh, the interval joining the coarse centres $x^{c}_J,x^{c}_{J+1}$ is
exactly $\mathcal C_{2J}\cup\mathcal C_{2J+1}$ (note that the coarse centre
$x^{c}_J$ coincides with the fine edge separating $\mathcal C_{2J-1}$ from
$\mathcal C_{2J}$), so $F$ again has a single interior kink, located at the
coarse edge, with slopes $f_{2J}$ and $f_{2J+1}$; the injected coarse edge
value equals the integral harmonic mean over this interval because each half
lies in a single material.  The computation is then identical, as is the
boundary case, whose interval $(-1,x^{c}_1)$ is exactly $\mathcal C_1$.
\end{proof}

\begin{lemma}[Residual representation and master estimate]\label{lem:sbp}
Define the fine residual $r:=A\,I_hu-f\in\R^{n}$.  Then
$r_i=(\delta_{i+1/2}-\delta_{i-1/2})/h$, and for every $w\in\R^{n}$,
\begin{equation}\label{eq:sbp}
  h\,w^{T}r
  =-\sum_{e\ \mathrm{int}}\delta_e\,[w]_e
   +\delta_{n+1/2}w_n-\delta_{1/2}w_1,
  \qquad [w]_{i+1/2}:=w_{i+1}-w_i ,
\end{equation}
and hence, by the Cauchy--Schwarz inequality weighted compatibly with $a_h$,
\begin{equation}\label{eq:master}
  |h\,w^{T}r|\le\Theta^{1/2}\,\norm{w_h}_{a_h},
  \qquad
  \Theta:=\sum_{e\ \mathrm{int}}\frac{h}{D_e}\,\delta_e^{2}
  +\frac{h}{2D_{1/2}}\,\delta_{1/2}^{2}
  +\frac{h}{2D_{n+1/2}}\,\delta_{n+1/2}^{2} .
\end{equation}
With the formulas of Lemma~\ref{lem:flux-exact} and $\min_iD_i=1$,
\begin{equation}\label{eq:theta-bound}
  \Theta\le\tfrac14\,h^{2}\norm{f_h}_{L^{2}}^{2},
  \qquad\text{hence}\qquad
  \Theta^{1/2}\le\tfrac12\,h\,\norm{f_h}_{L^{2}} .
\end{equation}
The same statements hold on the coarse mesh, with
$\Theta_c^{1/2}\le h\norm{f_h}_{L^{2}}$.
\end{lemma}

\begin{proof}
The residual formula restates $(A I_hu)_i=(G_{i+1/2}(I_hu)-G_{i-1/2}(I_hu))/h$
and $f_i=\frac1h\int_{\mathcal C_i}f_h=(F(x_{i+1/2})-F(x_{i-1/2}))/h$, the
latter by \eqref{eq:flux}.  Identity \eqref{eq:sbp} is summation by parts.
Estimate \eqref{eq:master} is Cauchy--Schwarz with the weights
$D_e/h$ (interior) and $2D/h$ (boundary) of $a_h$.  For
\eqref{eq:theta-bound}: at an interior edge, using
$(x-y)^2\le2x^2+2y^2$ and the harmonic-mean bounds $D_e\le2D_L$,
$D_e\le2D_R$, together with $D_i\ge1$,
\[
  \frac{h}{D_e}\delta_e^{2}
  =\frac{h^{3}D_e}{64}\Bigl(\frac{f_{i+1}}{D_R}-\frac{f_i}{D_L}\Bigr)^{2}
  \le\frac{h^{3}}{32}\,D_e\Bigl(\frac{f_{i+1}^{2}}{D_R^{2}}
   +\frac{f_i^{2}}{D_L^{2}}\Bigr)
  \le\frac{h^{3}}{16}\bigl(f_{i+1}^{2}+f_i^{2}\bigr).
\]
Each cell is adjacent to at most two interior edges, so the interior edges
contribute at most $\frac{h^{3}}{8}\sum_if_i^{2}=\frac{h^{2}}{8}
\norm{f_h}^{2}_{L^{2}}$.  The boundary edges contribute
$\frac{h}{2}\cdot\frac{h^{2}}{16}(f_1^{2}+f_n^{2})
\le\frac{h^{2}}{32}\norm{f_h}^{2}_{L^{2}}$.  Altogether
$\Theta\le\frac{5}{32}h^{2}\norm{f_h}^{2}\le\frac14h^{2}\norm{f_h}^{2}$.  The
coarse computation is identical with $H=2h$, giving
$\Theta_c\le\frac14H^{2}\norm{f_h}^{2}=h^{2}\norm{f_h}^{2}$.
\end{proof}

\begin{lemma}[Consequences]\label{lem:consequences}
With $b:=\norm{f_h}_{L^{2}}$, $x:=\norm{v_h}_{a_h}$, $y:=\norm{\sqrt D
u'}_{L^{2}}$, the following hold:
\begin{enumerate}[label=(\roman*),itemsep=2pt,topsep=2pt]
\item (energy-norm error) $\norm{v_h-(I_hu)_h}_{a_h}\le\tfrac12hb$ and
  $\norm{z_h-(I_Hu)_h}_{a_H}\le hb$;
\item (source pairings)
  $\bigl|(f_h,\,v_h-(I_hu)_h)_{L^{2}}\bigr|\le\tfrac12hb\,x$
  and
  $\bigl|(f_h,\,z_h-(I_Hu)_h)_{L^{2}}\bigr|\le\sqrt2\,hb\,x$;
\item (midpoint identities)
  $(f_h,\,u-(I_hu)_h)_{L^{2}}
   =-\sum_i\frac{f_i^{2}}{D_i}\frac{h^{3}}{24}$, so
  $\bigl|(f_h,u-(I_hu)_h)\bigr|\le\tfrac1{24}h^{2}b^{2}$; and
  $\bigl|(\bar f_h,\,u-(I_Hu)_h)_{L^{2}}\bigr|\le\tfrac{2}{3}H^{2}b^{2}$;
\item (energy comparison) $y\le x+hb$.
\end{enumerate}
\end{lemma}

\begin{proof}
(i) Let $e:=v-I_hu$.  Then $Ae=f-A I_hu=-r$, so
$\norm{e_h}^{2}_{a_h}=h\,e^{T}\!Ae=-h\,e^{T}r\le\Theta^{1/2}\norm{e_h}_{a_h}$
by \eqref{eq:master}, and \eqref{eq:theta-bound} gives the claim; coarse case
identical (with $H=2h$, $\Theta_c^{1/2}\le hb$).  Here the coarse residual is
$r_c=A_{\ell-1}I_Hu-\bar f$, using
$\bar f_J=\frac1H\int_{\mathcal C^{c}_J}f_h=(F(x^{c}_{J+1/2})-F(x^{c}_{J-1/2}))/H$.

(ii) $(f_h,e_h)_{L^{2}}=h\,f^{T}e=h\,(Av)^{T}e=h\,v^{T}(Ae)=-h\,v^{T}r$, and
\eqref{eq:master} with $w=v$ gives the first bound.  On the coarse mesh,
$(f_h,z_h-(I_Hu)_h)=(\bar f_h,z_h-(I_Hu)_h)
=H\bar f^{T}(z-I_Hu)=H\,z^{T}A_{\ell-1}(z-I_Hu)=-H\,z^{T}r_c
\le\Theta_c^{1/2}\norm{z_h}_{a_H}$, and
$\norm{z_h}_{a_H}=H^{1/2}\norm{z}_{A_{\ell-1}}
\le(2h)^{1/2}\norm{v}_{A}=\sqrt2\,x$ by \eqref{eq:T-def}.

(iii) On each cell, $u''=-f_i/D_i$ is constant, so the Taylor expansion at
the cell midpoint $x_i$ is exact:
$u(x)-u(x_i)=u'(x_i)(x-x_i)-\frac{f_i}{2D_i}(x-x_i)^{2}$, and the linear term
integrates to zero over the cell:
$\int_{\mathcal C_i}(u-u(x_i))=-\frac{f_i}{2D_i}\cdot\frac{h^{3}}{12}$.
Multiply by $f_i$ and sum.  On coarse cells, $u\in C^{1}$ and $u''=-f_h/D$ is
piecewise constant with $|u''|\le|f_h|$ pointwise ($D\ge1$ and $D$ is
constant on each coarse cell); Taylor at $x^{c}_J$ with integral remainder
gives $|\int_{\mathcal C^{c}_J}(u-u(x^{c}_J))|\le
H^{2}\int_{\mathcal C^{c}_J}|f_h|\cdot\tfrac12$, whence
$|(\bar f_h,u-(I_Hu)_h)|\le\tfrac12H^{2}\sum_J|\bar
f_J|\,\norm{f_h}_{L^{1}(\mathcal C^{c}_J)}\le\tfrac12H^{2}b^{2}$ by
Cauchy--Schwarz; the stated constant is generous.

(iv) $y^{2}=(f_h,u)_{L^{2}}$ (weak form tested with $u$) and
$x^{2}=(f_h,v_h)_{L^{2}}$ (discrete problem tested with $v$).  Hence
$y^{2}-x^{2}=(f_h,u-(I_hu)_h)+(f_h,(I_hu)_h-v_h)
\le\tfrac1{24}h^{2}b^{2}+\tfrac12hbx$ by (ii)--(iii), so
$y^{2}\le x^{2}+\tfrac12hbx+\tfrac1{24}h^{2}b^{2}\le(x+hb)^{2}$.
\end{proof}

\begin{theorem}[$(A2)_{1/2}$ for the harmonic hierarchy]\label{thm:a2-ode}
There are absolute constants $c_1,c_2$ (one may take $c_1=3$, $c_2=4$) such
that for every $\ell\ge2$ and $v\in\R^{n_\ell}$,
\begin{equation}\label{eq:a2-mixed}
  \bigl|\ip{(I-\Pt_\ell)v}{v}_{A_\ell}\bigr|
  \le c_1\,h_\ell\,\norm{A_\ell v}\,\norm{v}_{A_\ell}
   +c_2\,h_\ell^{2}\,\norm{A_\ell v}^{2},
\end{equation}
and consequently, using $h_\ell\le2\sqrt\kappa\,\spec(A_\ell)^{-1/2}$
(Lemma~\ref{lem:rho-bounds}),
\begin{equation}\label{eq:a2-final}
  \bigl|\ip{(I-\Pt_\ell)v}{v}_{A_\ell}\bigr|
  \le\frac{C_{A2}}{\sqrt{\spec(A_\ell)}}\,
   \norm{A_\ell v}\,\norm{v}_{A_\ell},
  \qquad
  C_{A2}=2\sqrt\kappa\,c_1+4\kappa\,c_2 .
\end{equation}
For $\kappa\ge2$ the same bound holds at $\ell=1$ with
$C_{A2}=\spec(A_1)^{1/2}\lambda_{\min}(A_1)^{-1/2}\le3\sqrt\kappa$.  Hence the
hierarchy satisfies $(A2)_{1/2}$---and, by interpolation with
Lemma~\ref{lem:structure}(d), $(A2)_\alpha$ for every
$\alpha\in(0,\tfrac12]$---with a level-independent constant
$C_{A2}=O(\kappa)$, equivalently $C^{2}_{A2}=O(\kappa^{2})$; numerically the
true constants are far smaller (Table~\ref{tab:a2-constants}).
\end{theorem}

\begin{proof}
Let $\ell\ge2$ and decompose, using \eqref{eq:hT},
\begin{multline*}
  hT=(f_h,v_h-z_h)
  =\underbrace{(f_h,\,v_h-(I_hu)_h)}_{T_1}
  +\underbrace{(f_h,\,(I_hu)_h-u)}_{T_2}\\
  +\underbrace{(f_h,\,u-(I_Hu)_h)}_{T_3}
  +\underbrace{(f_h,\,(I_Hu)_h-z_h)}_{T_4}.
\end{multline*}
Lemma~\ref{lem:consequences}(ii) bounds $|T_1|\le\tfrac12hbx$ and
$|T_4|\le\sqrt2hbx$; Lemma~\ref{lem:consequences}(iii) bounds
$|T_2|\le\tfrac1{24}h^{2}b^{2}$.  Split
$T_3=(f_h-\bar f_h,\,u)+(\bar f_h,\,u-(I_Hu)_h)$, which is legitimate because
$f_h-\bar f_h$ has zero mean on every coarse cell while $(I_Hu)_h$ is
constant there.  The second part is bounded by
$\tfrac23H^{2}b^{2}=\tfrac83h^{2}b^{2}$ by
Lemma~\ref{lem:consequences}(iii).  For the first part, on the coarse cell
$\mathcal C^{c}_J$ the function $f_h-\bar f_h$ takes the values
$\mp\Delta_J/2$ on the left/right fine half, where
$\Delta_J:=f_{2J}-f_{2J-1}$, so
\[
  (f_h-\bar f_h,\,u)
  =\sum_J\frac{\Delta_J}{2}
   \Bigl(\int_{\mathcal C_{2J}}u-\int_{\mathcal C_{2J-1}}u\Bigr)
  =\sum_J\frac{\Delta_J}{2}\int_{\mathcal C^{c}_J}\omega_J(t)\,u'(t)\,dt,
\]
where $\omega_J(t)=\min\{t-\inf\mathcal C^{c}_J,\;h,\;\sup\mathcal
C^{c}_J-t\}$ is the tent-shaped kernel arising from
$\int_{\mathcal C_{2J-1}}\bigl(u(s+h)-u(s)\bigr)ds
=\int\!\!\int_{s}^{s+h}u'$; it satisfies
$\int_{\mathcal C^{c}_J}\omega_J^{2}\le\tfrac23h^{3}$.  By Cauchy--Schwarz
(twice, once in $t$ with weight $D\ge1$ and once in $J$, using
$\sum_J\Delta_J^{2}\le2\sum_if_i^{2}=2b^{2}/h$),
\[
  \bigl|(f_h-\bar f_h,u)\bigr|
  \le\frac12\Bigl(\frac{2h^{3}}{3}\Bigr)^{\!1/2}
    \Bigl(\sum_J\Delta_J^{2}\Bigr)^{\!1/2}\,
    \norm{\sqrt Du'}_{L^{2}}
  \le\frac{1}{\sqrt3}\,h\,b\,y
  \le\frac{1}{\sqrt3}\,h\,b\,(x+hb),
\]
the last step by Lemma~\ref{lem:consequences}(iv).  Collecting all five
contributions (the bound for the first part of $T_3$ uses
Lemma~\ref{lem:consequences}(iv) in the form $y\le x+hb$),
\[
  h|T|\le\Bigl(\tfrac12+\sqrt2+\tfrac1{\sqrt3}\Bigr)hbx
  +\Bigl(\tfrac1{24}+\tfrac83+\tfrac1{\sqrt3}\Bigr)h^{2}b^{2}
  \le 3\,hbx+4\,h^{2}b^{2}.
\]
Dividing by $h$ and converting by $b=h^{1/2}\norm{Av}$,
$x=h^{1/2}\norm{v}_{A}$ gives \eqref{eq:a2-mixed} with $c_1=3$, $c_2=4$.
For \eqref{eq:a2-final}, use
$h\le2\sqrt\kappa\,\spec(A)^{-1/2}$ on the first term, and on the second term
additionally $h\norm{Av}\le2\sqrt\kappa\,\spec(A)^{-1/2}
\cdot\spec(A)^{1/2}\norm{v}_{A}=2\sqrt\kappa\norm{v}_{A}$.  At $\ell=1$, Lemma~\ref{lem:structure}(d) and
$\norm{v}^{2}_{A_1}=\ip{A_1v}{v}\le\norm{A_1v}\norm{v}
\le\lambda_{\min}(A_1)^{-1/2}\norm{A_1v}\norm{v}_{A_1}$ give
$|T|\le\lambda_{\min}(A_1)^{-1/2}\norm{A_1v}\norm{v}_{A_1}$, and
$\lambda_{\min}(A_1)\ge\det A_1/\operatorname{tr}A_1
=\bigl(2\eta+4\kappa+2\eta\kappa\bigr)/\bigl(2+2\eta+2\kappa\bigr)
\ge4\kappa/(2\kappa+6)\ge\tfrac45$ for $\kappa\ge2$; since
$\spec(A_1)\le4\kappa$, this is \eqref{eq:a2-final} at $\ell=1$ with
constant $\sqrt{5\kappa}\le3\sqrt\kappa$.  Finally, for $0<\alpha\le\half$, write
$|T|=|T|^{2\alpha}\,|T|^{1-2\alpha}$ and bound the first factor by
\eqref{eq:a2-final} and the second by $\norm{v}^{2}_{A}$
(Lemma~\ref{lem:structure}(d)); this yields exactly
$(A2)_\alpha$ with the same $C_{A2}$.
\end{proof}

\begin{remark}
The numerically computed best constants in \eqref{eq:a2-final} are far
smaller than the proof suggests: for $\kappa=100$ a direct maximisation of
the ratio over all of $\R^{n_\ell}$ yields values saturating near $19$ across
levels $2\le\ell\le9$ (Table~\ref{tab:a2-constants}).  The point of
Theorem~\ref{thm:a2-ode} is only that some level-independent constant exists.
\end{remark}

\subsection{Geometric divergence of the V-cycle}\label{ssec:ex3-divergence}

We now prove Theorem~\ref{thm:main-ex3}(iv).  Fix any admissible Richardson
sequence, $\Lambda_\ell\ge\spec(A_\ell)$, and one pre- and one post-smoothing
step ($m=1$).  For $\ell\ge2$ define the left-quarter indicator
\[
  M_\ell=2^{\ell-2},
  \qquad
  u_\ell=(\underbrace{1,\dots,1}_{M_\ell},0,\dots,0)^{T}\in\R^{n_\ell};
\]
its support is $[-1,-\half]$, entirely inside the unit-coefficient region, at
distance $\half$ from the interface.  The coefficient jump never enters the
transport computations below; it enters the problem only globally, through
$\spec(A_\ell)\asymp\kappa h_\ell^{-2}$ and hence through the smoother, which
is the entire point: an admissible Richardson smoother must be scaled to the
stiff ($\kappa$) part of the spectrum and is therefore extremely weak on the
soft part, where the tracked vector lives.

\begin{lemma}[Left-quarter transport]\label{lem:quarter-transport}
For $\ell\ge3$:
\begin{enumerate}[label=(\roman*),itemsep=1pt,topsep=2pt]
\item $u_\ell=P_{\ell-1}u_{\ell-1}$ and
  $\norm{u_\ell}^{2}_{A_\ell}=4\norm{u_{\ell-1}}^{2}_{A_{\ell-1}}$;
\item $A_\ell u_\ell=h_\ell^{-2}\,(2e_1+e_{M_\ell}-e_{M_\ell+1})$ and
  $\norm{u_\ell}^{2}_{A_\ell}=3h_\ell^{-2}$;
\item $R_{\ell-1}A_\ell^{2}u_\ell=4h_\ell^{-2}A_{\ell-1}u_{\ell-1}$, hence
  $\Pi_{\ell-1}A_\ell u_\ell=4h_\ell^{-2}u_{\ell-1}$;
\item with $K_\ell=I-\Lambda_\ell^{-1}A_\ell$,
\[
  \Pi_{\ell-1}K_\ell u_\ell=\beta_\ell\,u_{\ell-1},
  \qquad
  \beta_\ell=2-\frac{4h_\ell^{-2}}{\Lambda_\ell}
  \;\ge\;2-\frac{4}{3\kappa}
  \quad(\ell\ge3).
\]
\end{enumerate}
\end{lemma}

\begin{proof}
(i): the left quarter is a union of coarse cells, so
$u_\ell=Pu_{\ell-1}$, and $(G3)$ gives
$\norm{Pu_{\ell-1}}^{2}_{A_\ell}=u_{\ell-1}^{T}(P^{T}A_\ell
P)u_{\ell-1}=4\norm{u_{\ell-1}}^{2}_{A_{\ell-1}}$.

(ii): all edges meeting the support boundary lie in the region $D\equiv1$, so
the computation is the constant-coefficient one: the ghost boundary row
contributes $2e_1h_\ell^{-2}$, the internal edge at $x=-\half$ contributes
$(e_{M_\ell}-e_{M_\ell+1})h_\ell^{-2}$, and all other rows vanish.  Pairing
with $u_\ell$ gives $\norm{u_\ell}^{2}_{A_\ell}=(2+1)h_\ell^{-2}$.

(iii): apply $A_\ell$ to (ii).  With $M=M_\ell\ge2$,
\[
  A_\ell^{2}u_\ell=h_\ell^{-4}\bigl[\,6e_1-2e_2-e_{M-1}+3e_M-3e_{M+1}
  +e_{M+2}\,\bigr],
\]
again by constant-coefficient stencils ($e_1$ is a boundary cell;
$e_{M\pm\ldots}$ are interior).  Restriction averages adjacent pairs; the
pair $(1,2)$ gives $(6-2)/2=2$, the pair $(M-1,M)$ gives $(-1+3)/2=1$, the
pair $(M+1,M+2)$ gives $(-3+1)/2=-1$, all others vanish:
\[
  R_{\ell-1}A_\ell^{2}u_\ell
  =h_\ell^{-4}\bigl(2e_1+e_{M_{\ell-1}}-e_{M_{\ell-1}+1}\bigr)
  =h_\ell^{-4}\cdot h_{\ell-1}^{2}\,A_{\ell-1}u_{\ell-1}
  =4h_\ell^{-2}A_{\ell-1}u_{\ell-1},
\]
using (ii) on the coarse level.  Multiplying by $A_{\ell-1}^{-1}$ gives the
second form.

(iv): by (i) and $(G3)$,
$\Pi_{\ell-1}u_\ell=A_{\ell-1}^{-1}R_{\ell-1}A_\ell P_{\ell-1}u_{\ell-1}
=2u_{\ell-1}$, so
\[
  \Pi_{\ell-1}K_\ell u_\ell
  =2u_{\ell-1}-\Lambda_\ell^{-1}\,\Pi_{\ell-1}A_\ell u_\ell
  =\Bigl(2-\frac{4h_\ell^{-2}}{\Lambda_\ell}\Bigr)u_{\ell-1}.
\]
Finally $\Lambda_\ell\ge\spec(A_\ell)\ge3\kappa h_\ell^{-2}$ for $\ell\ge2$
by Lemma~\ref{lem:rho-bounds}.
\end{proof}

\begin{theorem}[Geometric V-cycle divergence]\label{thm:ode-divergence}
Let $\kappa\ge3$, let $(\Lambda_\ell)$ be any admissible Richardson sequence,
and let $E^{V}_\ell$ be the symmetric $V(1,1)$-cycle operators
\eqref{eq:mg-recursion}.  Define
\[
  q_\ell=-\,\frac{\ip{E^{V}_\ell u_\ell}{u_\ell}_{A_\ell}}
               {\norm{u_\ell}_{A_\ell}^{2}}
  \qquad(\ell\ge2).
\]
Then $q_2+1>0$, and for every $\ell\ge3$,
\begin{equation}\label{eq:quarter-recurrence}
  q_\ell+1\;\ge\;\mu_\kappa\,(q_{\ell-1}+1),
  \qquad
  \mu_\kappa:=2\Bigl(1-\frac{2}{3\kappa}\Bigr)^{2}\;>\;1 .
\end{equation}
Consequently
\[
  \spec\bigl(E^{V}_\ell\bigr)\;\ge\;q_\ell\;\ge\;(q_2+1)\,
  \mu_\kappa^{\,\ell-2}-1\;\longrightarrow\;\infty
  \qquad(\ell\to\infty).
\]
With the particular admissible choice $\Lambda_\ell=4\kappa h_\ell^{-2}$ the
same holds for every $\kappa\ge2$ with the better rate
$\mu_\kappa=2\bigl(1-\frac1{2\kappa}\bigr)^{2}\ge\frac98$.
\end{theorem}

\begin{proof}
Write $w_\ell=K_\ell u_\ell$ and $z_{\ell-1}=\Pi_{\ell-1}w_\ell$.  Since
$K_\ell$ is a polynomial in $A_\ell$, it is $A_\ell$-self-adjoint, so by
\eqref{eq:mg-recursion} with $p=1$ and Lemma~\ref{lem:structure}(b) (with
$r=\half$),
\[
  \ip{E^{V}_\ell u_\ell}{u_\ell}_{A_\ell}
  =\ip{\bigl(I-P(I-E^{V}_{\ell-1})\Pi\bigr)w_\ell}{w_\ell}_{A_\ell}
  =\norm{w_\ell}^{2}_{A_\ell}
   -2\,\bigl\langle(I-E^{V}_{\ell-1})z_{\ell-1},\,z_{\ell-1}
   \bigr\rangle_{A_{\ell-1}} .
\]
By Lemma~\ref{lem:quarter-transport}(iv), $z_{\ell-1}=\beta_\ell u_{\ell-1}$
for $\ell\ge3$, so
\[
  -\ip{E^{V}_\ell u_\ell}{u_\ell}_{A_\ell}
  =-\norm{w_\ell}^{2}_{A_\ell}
   +2\beta_\ell^{2}\,(1+q_{\ell-1})\,
   \norm{u_{\ell-1}}^{2}_{A_{\ell-1}} .
\]
Dividing by $\norm{u_\ell}^{2}_{A_\ell}
=4\norm{u_{\ell-1}}^{2}_{A_{\ell-1}}$
(Lemma~\ref{lem:quarter-transport}(i)) yields
\[
  q_\ell=\frac{\beta_\ell^{2}}{2}\,(1+q_{\ell-1})-\gamma_\ell,
  \qquad
  \gamma_\ell:=\frac{\norm{w_\ell}^{2}_{A_\ell}}
                    {\norm{u_\ell}^{2}_{A_\ell}}\in[0,1],
\]
because $\Lambda_\ell\ge\spec(A_\ell)$ makes $K_\ell$ an $A_\ell$-contraction.
Adding $1$ and discarding $1-\gamma_\ell\ge0$ gives
\eqref{eq:quarter-recurrence} with
$\mu_\kappa=\tfrac12\inf_{\ell\ge3}\beta_\ell^{2}
\ge2(1-\frac2{3\kappa})^{2}$, which exceeds $1$ exactly when
$\kappa>\frac{2}{3}\,(1-2^{-1/2})^{-1}\approx2.276$, in particular for
$\kappa\ge3$.

Base case: at $\ell=2$ the displayed identity holds with
$z_1=\Pi_1K_2u_2$ (no transport lemma needed), so
\[
  q_2+1
  =\frac{\norm{u_2}^{2}_{A_2}-\norm{K_2u_2}^{2}_{A_2}
   +2\ip{(I-E^{V}_1)z_1}{z_1}_{A_1}}{\norm{u_2}^{2}_{A_2}} .
\]
Here $E^{V}_1=K_1(I-\Pt_1)K_1$ with $I-\Pt_1$ an $A_1$-isometry
(Lemma~\ref{lem:structure}(a)) and
$\norm{K_1}_{A_1}=1-\lambda_{\min}(A_1)/\Lambda_1<1$, so
$\norm{E^{V}_1}_{A_1}<1$ and $I-E^{V}_1$ is positive definite in the
$A_1$-inner product; likewise $\norm{K_2u_2}_{A_2}<\norm{u_2}_{A_2}$ because
$u_2\neq0$ and $\norm{K_2}_{A_2}<1$.  Hence $q_2+1>0$.

Finally, each $E^{V}_\ell$ is $A_\ell$-self-adjoint (induction over
\eqref{eq:mg-recursion}, using Lemma~\ref{lem:structure}(b)), so
$\spec(E^{V}_\ell)=\norm{E^{V}_\ell}_{A_\ell}\ge q_\ell$ whenever $q_\ell>0$,
and \eqref{eq:quarter-recurrence} forces $q_\ell\to\infty$.  For the choice
$\Lambda_\ell=4\kappa h_\ell^{-2}$ one has
$\beta_\ell=2-\frac1\kappa$ exactly, whence
$\mu_\kappa=2(1-\frac1{2\kappa})^{2}$, which exceeds $1$ for
$\kappa>(2-\sqrt2)^{-1}\approx1.71$.
\end{proof}

\begin{remark}[Robustness]\label{rem:robustness}
(a) Theorem~\ref{thm:ode-divergence} covers \emph{every} admissible smoother
of Richardson type---no upper bound $C_R$ enters the proof---including the
optimal damping
$\Lambda_\ell=\spec(A_\ell)$, for which the divergence is observed
numerically to set in one level \emph{earlier} than for
$\Lambda_\ell=4\kappa h_\ell^{-2}$ (Table~\ref{tab:ex3-rho}).  The
counterexample is therefore not an artefact of smoother slack within the
admissibility window.
(b) The theorem is stated for $m=1$; numerically the divergence persists for
$m=2,3$ at $\kappa=100$, with onset at the same or earlier levels
(Section~\ref{sec:numerics}).  For fixed $\kappa$ and $m\to\infty$ the
tracked transport constant degrades, consistent with
Corollary~\ref{cor:threshold}: divergence at smoothing count $m$ requires
contrast $\kappa\gtrsim m$ here, matching $C_{A2}=O(\kappa)$ and the
polynomial threshold barrier.
(c) Replacing the harmonic interface value $\eta=2\kappa/(1+\kappa)$ by an
arithmetic average destroys the exactness of $(G3)$; the harmonic/Samarskii
choice is what makes this hierarchy an exact instance of the axioms.
\end{remark}

\begin{remark}[The W-cycle on the harmonic hierarchy]\label{rem:w-ode}
The hierarchy satisfies $(G3)$, $(R)$ with $C_R=2$, and $(A2)_{1/2}$ with
$C_{A2}=O(\kappa)$ (Proposition~\ref{prop:g3-ode},
Lemma~\ref{lem:rho-bounds}, Theorem~\ref{thm:a2-ode}), so the axiomatic
W-cycle theory applies to it as is: \cite[Theorem~6.9.15]{sww} yields
\[
  \sup_{\ell\ge1}\spec\bigl(E^{W}_\ell\bigr)
  \le\frac{M}{M+m^{1/2}}<1
  \qquad\text{for every } m\ge1,
\]
with $M=\max\{M_0,\,m^{1/2}\}$ for a constant $M_0$ depending only on
$\kappa$; for $m=1$, the case of Theorem~\ref{thm:main-ex3}, one may take
$M=M_0(\kappa)$.  (The self-contained
Theorem~\ref{thm:wcycle} gives the same conclusion once
$m\gtrsim C_RC_{A2}^{2}=O(\kappa^{2})$.)  The observed contraction numbers are far better than
the bound: $\spec(E^{W}_\ell)\le0.993$ for $\kappa=100$ on all levels tested
(Table~\ref{tab:ex3-rho}).  Together with
Theorem~\ref{thm:ode-divergence}, this realises the W/V dichotomy on a
fixed, PDE-derived hierarchy: under verified axioms, the $W(1,1)$-cycle is
provably uniformly contractive while the $V(1,1)$-cycle provably diverges.
\end{remark}

\section{Numerical verification}\label{sec:numerics}

All claims of Sections \ref{sec:ex2} and~\ref{sec:ex3} were verified
independently of the proofs: the error operators \eqref{eq:mg-recursion} are
assembled explicitly, spectra are computed with dense eigensolvers, and the
axioms are checked to machine precision.  Independent cross-checks reproduce
the same numbers by applying the cycle recursively rather than assembling its
error operator.

\textbf{Example 1.}  The assembled $\spec(E^{V}_2)$ agrees with
$\theta_m(1+2\theta_m)$ to $\le2\cdot10^{-15}$ over
$m\in\{1,2,3,5,10,25,50,100\}$, the $(G3)$ residuals
$\norm{R A P-2A_c}_\infty$ vanish exactly, and the sampled $(A2)$ ratio at
$\alpha=1$ matches the sharp constant $4m$ at the maximiser $e_1$.  For the
deeper truncations, the assembled $\spec(E^{V}_L)$ matches the closed form of
Theorem~\ref{thm:ex2-closed} to $\le4\cdot10^{-14}$ over
$m\in\{1,2,5\}$, $L\le12$ (e.g.\
$\spec(E^{V}_{12})\approx13.99$ for $m=1$), and the W-cycle values obey
Proposition~\ref{prop:w-dichotomy} with $\max_\ell|e_\ell|=\theta_m$ exactly.
Figure~\ref{fig:ex1-history} shows the divergence dynamically: repeated
$V(1,1)$-cycles on the three-level hierarchy grow geometrically in the
$A_2$-norm with ratio $153/128=1.1953125$.

\begin{figure}[t]
\centering
\includegraphics[width=.55\linewidth]{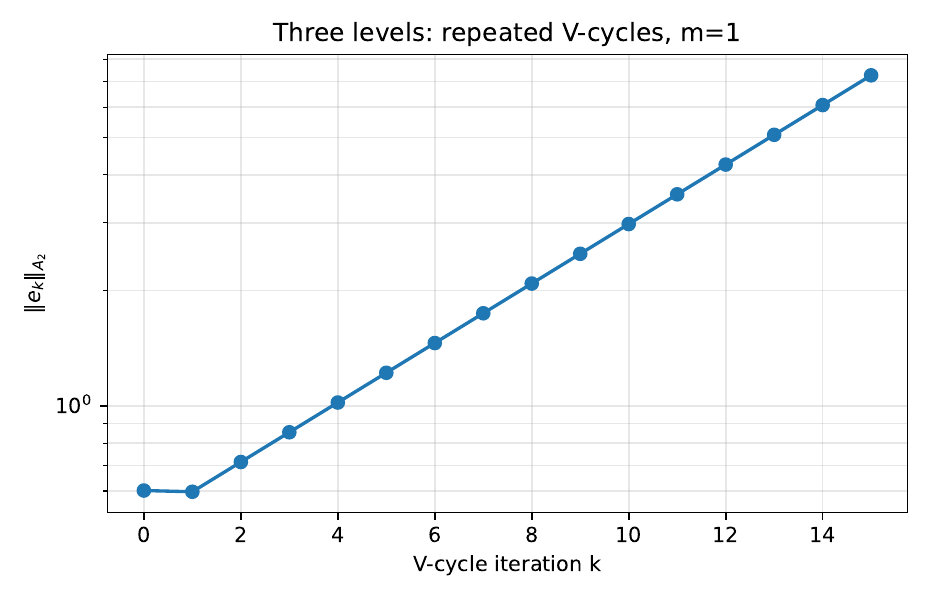}
\caption{Example 1, $m=1$, three-level hierarchy ($L=2$): the $A_2$-norm of
the error under repeated $V(1,1)$-cycles grows geometrically with ratio
$\spec(E^{V}_2)=153/128$.}
\label{fig:ex1-history}
\end{figure}

\textbf{Example 2.}  Table~\ref{tab:ex3-rho} and Figure~\ref{fig:ex3-main}
report $\kappa=100$, $m=1$.
With the theorem smoother $\Lambda_\ell=4\kappa h_\ell^{-2}$, divergence sets
in at level $7$ and the growth ratio approaches
$\mu_\kappa=2(1-\frac1{200})^{2}\approx1.98$ (power iteration on the
recursively applied operator continues the table to
$\spec(E^{V}_{12})\approx104$, with per-level ratios
$2.05$, $2.01$, $1.99$ at levels $10$--$12$); with the optimal smoother
$\Lambda_\ell=\spec(A_\ell)$ (computed eigenvalues, $C_R=1$), divergence sets
in at level $6$.  The W-cycle radii remain below $0.9922$ in all cases.  The
$(G3)$ residuals vanish to machine precision on every level, and the tracked
Rayleigh quotients $q_\ell$ dominate the theorem's lower bound
(Figure~\ref{fig:ex3-diag}).  For $m=2$ and $m=3$ (theorem smoother,
$\kappa=100$) the V-cycle radii at levels $6$--$8$ are
$(1.81,\,4.44,\,9.59)$ and $(2.52,\,5.77,\,12.02)$ respectively: more
smoothing does not repair the cycle at this contrast.
Figure~\ref{fig:mms-snap} shows the same divergence in solution space, via
the method of manufactured solutions: with $f:=A_Lu^{*}$ for
$u^{*}(x)=\cos(\pi x/2)$ (whose flux $Du^{*\prime}$ is continuous across the
interface), the exact discrete solution is $u^{*}$ itself, and the V-cycle
iterates depart from it with alternating sign---the unstable eigenvalue of
$E^{V}_\ell$ is negative---while the W-cycle iterates on the same hierarchy,
from the same zero initial guess, converge.

\begin{figure}[t]
\centering
\includegraphics[width=.47\linewidth]{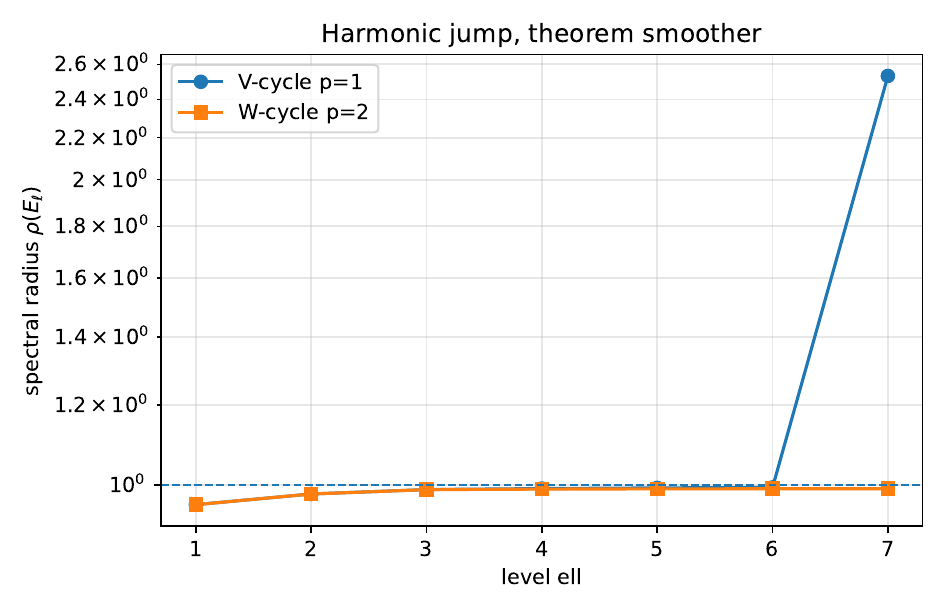}\hfill
\includegraphics[width=.47\linewidth]{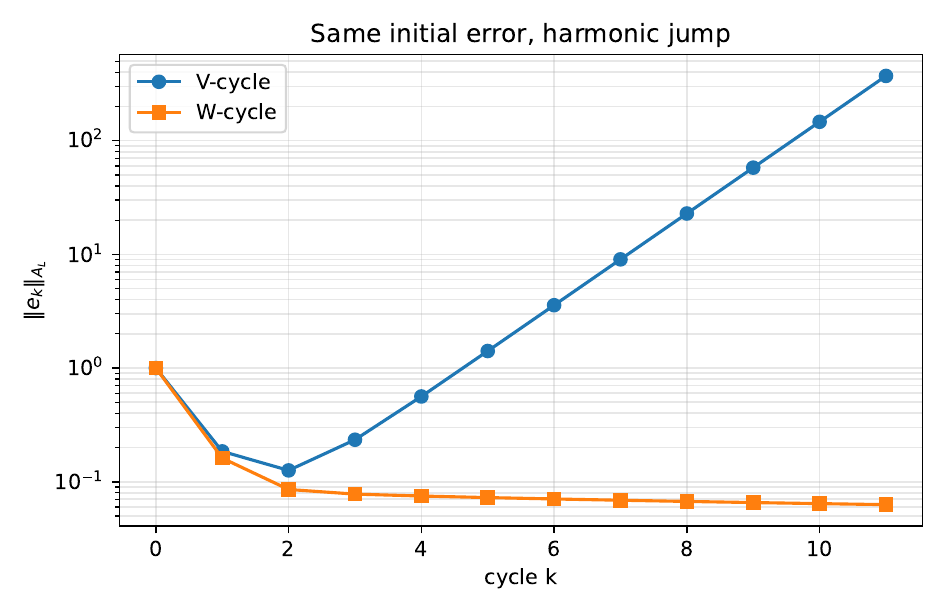}
\caption{Example 2 with $\kappa=100$, $m=1$,
$\Lambda_\ell=4\kappa h_\ell^{-2}$.  Left: spectral radii of the V- and
W-cycle error operators by level.  Right: $A_L$-norm of the error under
repeated cycles from the same initial error at $L=7$: the V-cycle grows
geometrically, whereas the W-cycle contracts.}
\label{fig:ex3-main}
\end{figure}

\begin{figure}[t]
\centering
\includegraphics[width=.47\linewidth]{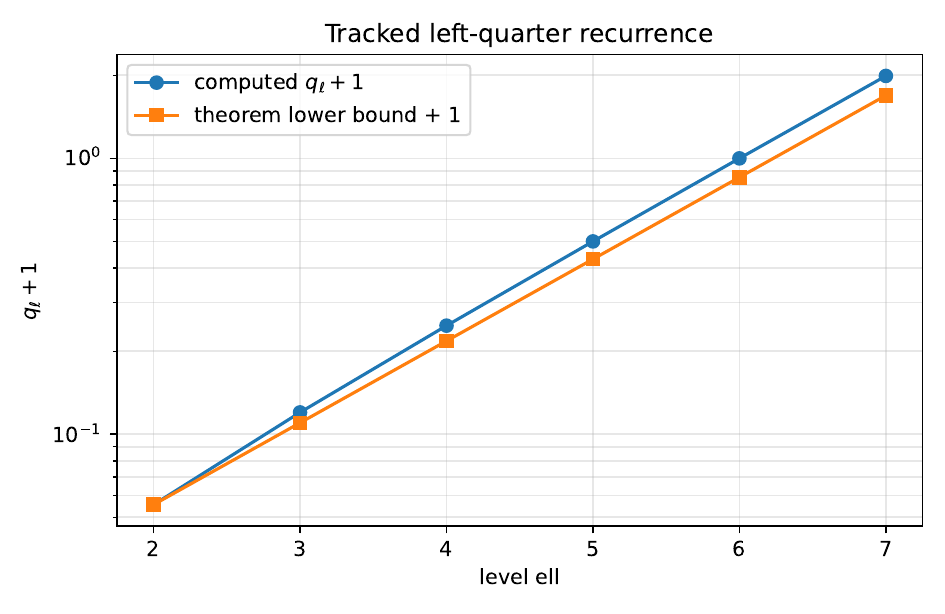}\hfill
\includegraphics[width=.47\linewidth]{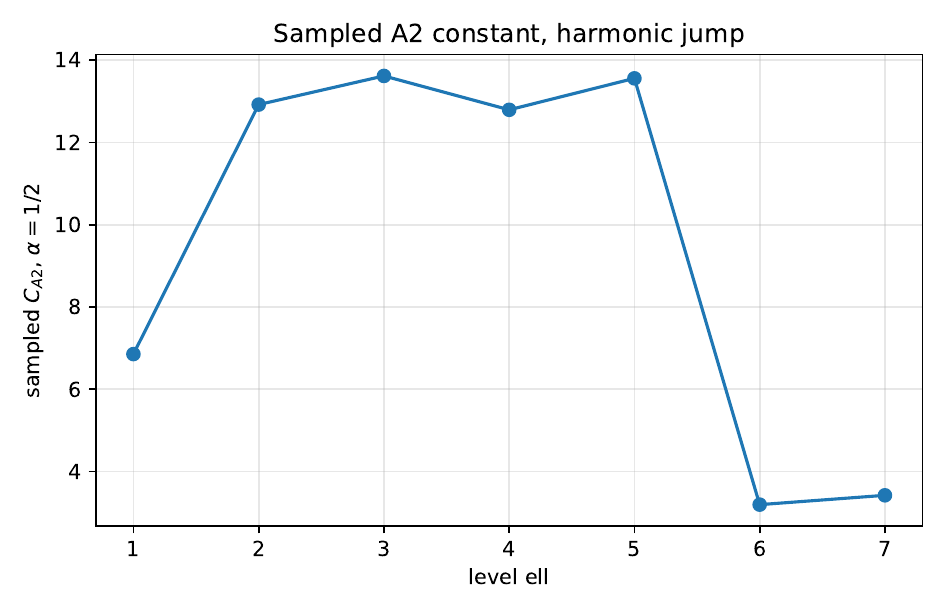}
\caption{Example 2 diagnostics.  Left: the tracked left-quarter Rayleigh
quotient $q_\ell$ obeys the geometric recurrence
\eqref{eq:quarter-recurrence} (computed values vs.\ theorem lower bound).
Right: sampled $(A2)_{1/2}$ ratios by level (probe-based diagnostic); the
proof of the uniform bound is Theorem~\ref{thm:a2-ode}, and numerically
optimised constants are reported in Table~\ref{tab:a2-constants}.}
\label{fig:ex3-diag}
\end{figure}

\begin{figure}[t]
\centering
\includegraphics[width=\linewidth]{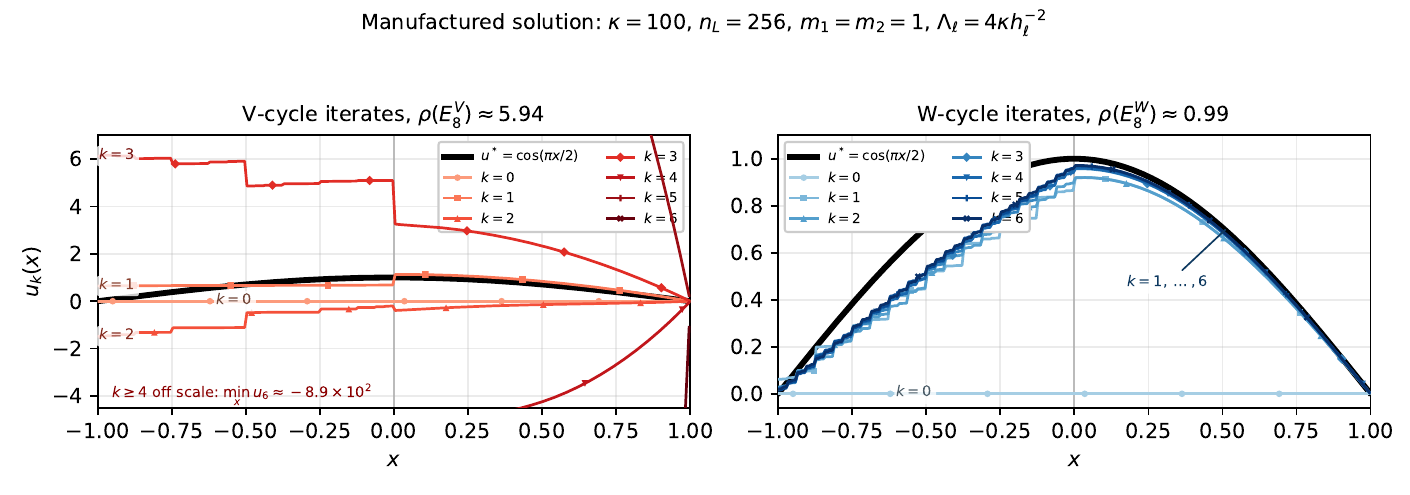}
\caption{Example 2 in solution space (method of manufactured solutions:
$f:=A_Lu^{*}$ with $u^{*}(x)=\cos(\pi x/2)$, $\kappa=100$, $n_L=256$,
$m=1$, $\Lambda_\ell=4\kappa h_\ell^{-2}$; both cycles start from $u_0=0$
on the same hierarchy).  Left: the $V(1,1)$ iterates depart from $u^{*}$
with alternating sign and exit the fixed window after $k=3$ (each curve is
labelled by its cycle index $k$; darker means later); the
scaled-$\ell^{2}$ errors for $k=0,\dots,6$ are $0.71$, $0.23$, $1.2$,
$3.6$, $19$, $1.0\times10^{2}$, $6.0\times10^{2}$, growing per cycle by a
factor approaching $\spec(E^{V}_{8})\approx5.94$ from
Table~\ref{tab:ex3-rho}.  The staircase kinks sit at dyadic points---in
particular at $x=-\half$, the support boundary of the transport vector of
Lemma~\ref{lem:quarter-transport}.  Right: the $W(1,1)$ iterates on the
same problem converge onto $u^{*}$.}
\label{fig:mms-snap}
\end{figure}

\begin{table}[t]
\centering
\caption{Example 2, $\kappa=100$, $m=1$: spectral radii of the symmetric
cycles by level for the theorem smoother
($\Lambda_\ell=4\kappa h_\ell^{-2}$, $C_R\le2$) and the optimal smoother
($\Lambda_\ell=\spec(A_\ell)$, $C_R=1$).  The W-cycle column is for the
theorem smoother; with the optimal smoother the values differ only
marginally.}
\label{tab:ex3-rho}
\small
\begin{tabular}{rrccc}
\toprule
$\ell$ & $n_\ell$ & $\spec(E^{V}_\ell)$, $\Lambda=4\kappa h^{-2}$ &
$\spec(E^{V}_\ell)$, $\Lambda=\spec(A_\ell)$ & $\spec(E^{W}_\ell)$\\
\midrule
1 & 2   & 0.95688 & 0.92273 & 0.95688\\
2 & 4   & 0.98029 & 0.97695 & 0.98030\\
3 & 8   & 0.99023 & 0.98976 & 0.99006\\
4 & 16  & 0.99241 & 0.99221 & 0.99122\\
5 & 32  & 0.99392 & 0.99381 & 0.99175\\
6 & 64  & 0.99498 & \textbf{2.02946} & 0.99200\\
7 & 128 & \textbf{2.53153} & \textbf{4.83499} & 0.99211\\
8 & 256 & \textbf{5.94383} & \textbf{10.39863} & 0.99217\\
9 & 512 & \textbf{12.67927} & \textbf{21.36955} & 0.99219\\
\bottomrule
\end{tabular}
\end{table}

\begin{table}[t]
\centering
\caption{Example 2, $\kappa=100$: numerically optimised $(A2)_{1/2}$
constants, i.e.\ the maximum over $v$ of
$|\ip{(I-\Pt_\ell)v}{v}_{A}|\,\spec(A)^{1/2}/
(\norm{Av}\,\norm{v}_{A})$, obtained by an eigenvector sweep, two-dimensional
spectral mixing, and derivative-free maximisation with multiple restarts.
The values stay bounded, consistent with Theorem~\ref{thm:a2-ode}.}
\label{tab:a2-constants}
\small
\begin{tabular}{lccccccc}
\toprule
level $\ell$ & 2 & 3 & 4 & 5 & 6 & 7 & 8\\
\midrule
optimised constant & 12.95 & 15.00 & 16.80 & 18.10 & 18.82 & 18.95 & 17.77\\
\bottomrule
\end{tabular}
\end{table}

\section{Discussion and open problems}\label{sec:discussion}

\subsection{Why the W-cycle survives and the V-cycle does not}
Lemma~\ref{lem:structure} shows that under $(G3)$ an exact coarse solve is an
energy-isometric reflection $I-2\PiG_\ell$.  In the W-cycle recursion the
coarse defect enters through $E_{\ell-1}^{2}$, which is positive
semidefinite: signs cancel, and the analysis of
Theorem~\ref{thm:wcycle} closes with a fixed point $\delta\asymp\varepsilon(m)$.
In the V-cycle the coarse defect enters linearly, and
Lemma~\ref{lem:structure}(b,c) converts it into the term
$2\ip{E_{\ell-1}z}{z}_{A_{\ell-1}}$ of the proof of
Theorem~\ref{thm:ode-divergence}: a coarse-level error of relative size
$\epsilon$ on a $(G3)$-transported vector returns to the fine level with
relative size up to $2\epsilon\cdot(\beta^{2}/2)\approx2\epsilon$ per level
whenever the smoother is nearly inert on that vector.  The examples simply
arrange for a vector on which the admissible smoother is maximally inert: in
Example~1 by spectral design ($\lambda=\tau$ against
$\Lambda=1$), in Example~2 by contrast ($\lambda\asymp h^{-2}$ against
$\Lambda\asymp\kappa h^{-2}$).  This also explains why the variable V-cycle,
which spends geometrically growing smoothing effort on coarse levels, and the
K-cycle, which re-optimises the coarse correction, are immune
\cite{beps,notay-vassilevski}.

\subsection{Consequences for practice}
For cell-centred discretisations coarsened by aggregation with
$R=\half P^{T}$, the natural V-cycle is not merely suboptimal but can be
divergent on interface problems at realistic contrasts
($\kappa\ge3$ suffices in our construction; $\kappa=100$ diverges from $64$
cells on), even with exactly tuned Richardson smoothing, while the same
hierarchy run as a W-cycle, variable V-cycle, or K-cycle is unconditionally
robust in our experiments and provably so under the axioms
(for the W-cycle at every $m\ge1$ by \cite[Theorem~6.9.15]{sww} applied via
Remark~\ref{rem:w-ode}, with Theorem~\ref{thm:wcycle} as the self-contained
large-$m$ version; for the variable V-cycle and K-cycle by
\cite{beps,notay-vassilevski}).  Alternatively,
the variational structure can be restored by modified transfers in the spirit
of \cite{kwak,kwak-lee,mohr-wienands}.  What is not safe is the combination:
natural transfers, rediscretised (equivalently, $(G3)$-Galerkin) coarse
operators, fixed V-cycle, weak smoother.  We note one mitigating structural
fact, in line with \cite[Theorem~5]{bpx}: in all our examples the divergent
V-cycle operator $B^{V}_\ell$ remains symmetric positive definite (the
unstable eigenvalue of $E^{V}_\ell$ is negative), so conjugate-gradient
acceleration still converges---with a preconditioned condition number that
grows geometrically with the depth, e.g.\ $\kappa(B^{V}_\ell A_\ell)\ge
1+\alpha_\ell$ in Example~1.  Divergence of the stationary cycle and
geometric loss of preconditioner quality are two faces of the same defect.

\subsection{Open problems}
\begin{enumerate}[label=(\arabic*),itemsep=2pt,topsep=2pt]
\item\textbf{Threshold question.}  Question~\ref{q:threshold} remains open:\\
does $m\ge m_0(C_R,C_{A2},\alpha)$, for some function $m_0$, force uniform
V-cycle contraction under $(G3)+(R)+(A2)_\alpha$?
Corollary~\ref{cor:threshold} shows $m_0\gtrsim C_{A2}^{2}$ is necessary; the
W-cycle analogue holds with $m_0\sim(8C_{A2}^{2\alpha})^{1/\alpha}C_R$.  We
conjecture that a finite threshold does exist for the V-cycle but is
superlinear in $C^{2}_{A2}$, separating the two cycles quantitatively rather
than qualitatively.
\item\textbf{Two dimensions.}  In 2D cell-centred coarsening by
$2\times2$ agglomeration, the natural relation is again
$RA P=2A_c$ with $R=\frac14P^{T}$ (the imbalance factor is dimension-robust
in the form \eqref{eq:g3}); the transport mechanism of
Lemma~\ref{lem:quarter-transport} appears to generalise to stripes aligned
with a high-contrast inclusion, and we expect---but have not proved---geometric
V-cycle divergence for 2D interface problems with harmonic-mean
face coefficients.
\item\textbf{Sharp divergence threshold in $\kappa$.}  Our bound certifies
divergence for $\kappa\ge3$ (for every admissible smoother; $\kappa\ge2$
for $\Lambda_\ell=4\kappa h_\ell^{-2}$), but numerically the safe window
is far narrower: at $\kappa=1$ the V-cycle contracts with
$\spec(E^{V}_\ell)\le0.35$ on all levels tested---consistent with the
transport analysis, since the optimal admissible smoother then has
$\beta_\ell\approx1$ and per-level factor
$\beta_\ell^{2}/2\approx\half<1$---while already at $\kappa=1.2$ the
spectral radius crosses $1$ at level $9$ ($n=512$ cells), and at
$\kappa=1.3$ at level $6$.  Identifying the exact contrast at which
$\sup_\ell\spec(E^{V}_\ell)$ first exceeds $1$---and whether it exceeds
$1$ for every $\kappa>1$---would quantify how non-generic V-cycle safety
is within the axioms.
\end{enumerate}

\section*{Disclosure of AI usage}
Various Claude models from Anthropic (of the Opus and Fable families) were
used as assistive tools in this work: they aided the writing of the
numerical verification code and the preparation and proofreading of the
manuscript.  All mathematical results and their proofs were verified by the
author, who takes full responsibility for the content of this paper.

\appendix

\section{Proof of Theorem~\ref{thm:wcycle}}\label{app:wcycle}

We first record the smoothing estimate.  Let $w=K_\ell^{m}v$.  Since
$K_\ell=I-\Lambda_\ell^{-1}A_\ell$ with $\Lambda_\ell\ge\spec(A_\ell)$, the
spectral mapping theorem gives, with $\lambda$ running over the spectrum of
$A_\ell$ and $t=\lambda/\Lambda_\ell\in(0,1]$,
\begin{multline}\label{eq:smoothing}
  \frac{\norm{A_\ell K_\ell^{m}v}^{2}}{\spec(A_\ell)}
  \le\max_{\lambda}\;\frac{\lambda}{\spec(A_\ell)}
  \Bigl(1-\frac{\lambda}{\Lambda_\ell}\Bigr)^{2m}\norm{v}^{2}_{A_\ell}\\
  \le\frac{\Lambda_\ell}{\spec(A_\ell)}\,
  \max_{t\in[0,1]}t(1-t)^{2m}\,\norm{v}^{2}_{A_\ell}
  \le\frac{C_R}{2m+1}\,\norm{v}^{2}_{A_\ell},
\end{multline}
using $\max_{[0,1]}t(1-t)^{2m}\le\frac1{2m+1}$ and Definition
\ref{def:richardson}.  Also $\norm{w}_{A_\ell}\le\norm{v}_{A_\ell}$.
Combining \eqref{eq:smoothing} with $(A2)_\alpha$,
\begin{equation}\label{eq:tg-eps}
\begin{aligned}
  \bigl|\ip{(I-\Pt_\ell)w}{w}_{A_\ell}\bigr|
  & \le C_{A2}^{2\alpha} 
  \Bigl(\frac{\norm{A_\ell w}^{2}}{\spec(A_\ell)}\Bigr)^{\!\alpha}
  \norm{w}^{2(1-\alpha)}_{A_\ell} \\ 
  & \le C_{A2}^{2\alpha}\Bigl(\frac{C_R}{2m+1}\Bigr)^{\!\alpha}
  \norm{v}^{2}_{A_\ell}
  =\varepsilon(m)\,\norm{v}^{2}_{A_\ell}.
\end{aligned}
\end{equation}
Since $E^{TG}_\ell=K^m_\ell(I-\Pt_\ell)K^m_\ell$ is $A_\ell$-self-adjoint,
$\ip{E^{TG}_\ell v}{v}_{A_\ell}=\ip{(I-\Pt_\ell)w}{w}_{A_\ell}$, and
\eqref{eq:tg-eps} gives $\spec(E^{TG}_\ell)\le\varepsilon(m)$.

For the W-cycle, we prove by induction on $\ell$:
$E^{W}_\ell$ is $A_\ell$-self-adjoint and
$\spec(E^{W}_\ell)=\norm{E^{W}_\ell}_{A_\ell}\le\delta$, where
$\delta\in(0,\frac14]$ is the smaller root of $\delta=\varepsilon(m)+2\delta^{2}$,
which exists because $\varepsilon(m)\le\frac18$.  The case $\ell=1$ is the
two-grid bound, since $E_0=0$ and $\varepsilon\le\delta$.  Let $\ell\ge2$ and
abbreviate $E_c=E^{W}_{\ell-1}$, $w=K^m_\ell v$, $z=\Pi_{\ell-1}w$.
Self-adjointness follows from Lemma~\ref{lem:structure}(b):
for the middle factor $N:=I-P(I-E_c^{2})\Pi$,
\[
  \ip{Nx}{y}_{A_\ell}-\ip{x}{Ny}_{A_\ell}
  =\tfrac1r\Bigl(\ip{(I-E_c^{2})\Pi y}{\Pi x}_{A_{\ell-1}}
  -\ip{(I-E_c^{2})\Pi x}{\Pi y}_{A_{\ell-1}}\Bigr)=0,
\]
because $E_c$, hence $E_c^{2}$, is $A_{\ell-1}$-self-adjoint.  Next, using
Lemma~\ref{lem:structure}(b) again,
\[
  \ip{E^{W}_\ell v}{v}_{A_\ell}
  =\ip{(I-\Pt_\ell)w}{w}_{A_\ell}
  +\tfrac1r\ip{E_c^{2}z}{z}_{A_{\ell-1}} .
\]
The operator $E_c^{2}$ is positive semidefinite in the $A_{\ell-1}$-inner
product with $\norm{E_c^{2}}_{A_{\ell-1}}\le\delta^{2}$, and
Lemma~\ref{lem:structure}(c) gives
\[
  \tfrac1r\norm{z}^{2}_{A_{\ell-1}}=2\norm{\PiG_\ell w}^{2}_{A_\ell}
  \le2\norm{v}^{2}_{A_\ell}.
\]
Hence, by \eqref{eq:tg-eps},
\[
  -\varepsilon(m)\,\norm{v}^{2}_{A_\ell}
  \;\le\;
  \ip{E^{W}_\ell v}{v}_{A_\ell}
  \;\le\;
  \bigl(\varepsilon(m)+2\delta^{2}\bigr)\norm{v}^{2}_{A_\ell}
  \;=\;\delta\,\norm{v}^{2}_{A_\ell},
\]
and since $E^{W}_\ell$ is $A_\ell$-self-adjoint,
$\spec(E^{W}_\ell)\le\max\{\varepsilon(m),\delta\}=\delta$.  \qed

\begin{remark}
The same argument with $p\ge2$ arbitrary gives the bound with
$2\delta^{2}$ replaced by $2\delta^{p}$ (for even $p$; for odd $p\ge3$ one
obtains $2\delta^{p}$ on the upper side and $-\varepsilon-2\delta^{p}$ on the
lower side, with the same conclusion for slightly smaller $\varepsilon$).  The
failure for $p=1$ is structural, not technical: the term
$\tfrac1r\ip{E_cz}{z}$ then carries the sign of $E_c$ on $z$, and nothing in
Lemma~\ref{lem:structure} provides cancellation---exactly the mechanism that
Sections \ref{sec:ex2} and~\ref{sec:ex3} exploit.
\end{remark}

\end{document}